%% file: main_v1.tex
\documentclass{./Article-eng}
\usepackage{pifont}

\usepackage{combelow}

\title{An abundance-type result for the tangent bundles of smooth Fano varieties\\
\textcolor{red}{Important Notice (Update as of December 3, 2025):
A gap has been identified in the Step 2 of the proof of Lemma 4.3 in the initial version (v1) of this article, thus invalidates the proof of the main theorem; an erratum is added to the end of the paper to explain this error. The author is currently preparing a corrected version, which is expected to be completed within approximately one month. Until the revision is available, readers are advised to exercise caution when referring to this work. The author sincerely apologizes for any inconvenience caused.}}
\author{Juanyong Wang}
\date{}

\begin{document}

\maketitle

\input{abstract}
\tableofcontents

%\input{intro}
%\input{preliminaries}
%\input{proof_A}
%\input{proof_B}

%%%%%%%%%%%%%%%%%%%%%%%%%%%%%%%%%%%%%%%%%%%%%%
%%\input{stability_tangent}

\input{V1/intro_v2}
\input{V1/preliminaries}
\input{V1/setup}

\input{V1/codim_B+}

\input{V1/cover_rc}

\input{V1/bigness_tangent}

%%\input{further}
%%%%%%%%%%%%%%%%%%%%%%%%%%%%%%%%%%%MMMMMMMMMMMM

%%%%%%%%%%%%%%%%%%%%%%%%%%%%%%%%%%%
%\input{Introduction}
%\input{Positivity-VectorBundle}
%\input{Stability-VectorBundle}
%\input{Generically-Nef}
%\input{Characterisation_Subbundle}

%%%%%%%%%%%
%\input{Reg-Foliation_Horing-Lemma}
%\input{Miyaoka-thm_Semistability-Tangent}
%\input{Flat-Direct-Images_Decomposition}
%%%%%%%%%%%%%%%%%%%%%%%%%%%%%%%%%%%%

\bibliography{CP}

\include{V1/erratum}

\end{document}

%% file: abstract.tex
\paragraph{Abstract:}
In this paper we prove the following abundance-type result: for any smooth Fano variety $X$, the tangent bundle $T_X$ is nef if and only if it is big and semiample in the sense that the tautological line bundle $\scrO_{\PP T_X}(1)$ is so, by which we establish a weak form of the Campana-Peternell conjecture (Camapan-Peternell, 1991). 


%Minimal Model Program;  augmented base locus;  rational curve; multiplier ideal 

%\begin{keywords}

%\end{keywords}

%% file: V1/intro_v2.tex
\section*{Introduction}
Throughout this article, we work over the field of complex numbers $\CC$. 
A \emph{fibre space} or a \emph{fibration} is a proper morphism between normal varieties that has connected fibres, and a \emph{$\PP^1$-fibration} is a smooth fibration whose fibres are all isomorphic to $\PP^1$ (cf. \cite[II.2.5, p.~105]{Kollar96}).

Inspired by the work of Mori \cite{Mor79} and of Siu-Yau \cite{SY80} on the Hartshorne-Frankel conjecture and by the works of Howard-Smyth-Wu, Mok and Zhong \cite{HSW81,MZ86,Mok88} on the generalised analytic version of the H-F conjecture, Campana and Peternell proposed in \cite{CP91} to establish the algebro-geometric counterpart of Mok's result, that is, the classification the compact Kähler manifolds with nef tangent bundle. 

Demailly, Peternell and Schneider study the Albanese maps of such Kähler manifolds in \cite{DPS94} and they proved that up to an étale cover the Albanese map of a compact Kähler manifold with nef tangent bundle is a locally trivial fibration, whose fibres are smooth Fano varieties with nef tangent bundles. Hence the problem is reduced to the study of such Fano varieties. In fact, in \cite{CP91} Campana and Peternell raised the following conjecture: 
%which characterizes rational homogeneous varieties by a numerical condition
\begin{conj}[{Campana-Peternell conjecture}]
\label{conj_CP}
Smooth Fano varieties with nef tangent bundles are rational homogeneous, in other word, they are quotients of semi-simple linear algebraic groups by parabolic subgroups. 
\end{conj}  

The {\hyperref[conj_CP]{Campana-Peternell conjecture}} is confirmed for surfaces and threefolds in \cite{CP91} by using the Mori-Mukai classification; the fourfold case is established by combining \cite{CP93} and \cite{Mok02,Hwa06}. However, the higher dimensional cases are still open except the following special cases:
\begin{itemize}
\item if $T_X$ is big and $1$-ample (\cite{SCW04});
\item if the Picard number of $X$ is very large compared to $\dim X$ (\cite{Wat15,Kan16});
\item if $X$ is of flag type, i.e. every elementary contraction of $X$ is a $\PP^1$-fibration (\cite{MOSCW15,OSCWW17}).
\end{itemize}
See \cite{MOSCWW15} for a quite complete survey of recent progress in this study. Following the notation in \cite{MOSCWW15}, in the sequel we will call a smooth Fano variety with nef tangent bundle a \emph{Campana-Peternell manifold} (\emph{CP-manifold} for short).

On the other hand, this conjecture can also be regarded among the \emph{`abundance-type conjectures'}, i.e. they predict that a weak positivity condition (e.g. nefness) for the canonically defined vector bundles (e.g. canonical bundles, tangent bundles, etc.) of a compact K\"ahler manifold can imply a strictly stronger positivity condition (e.g. semiampleness). Indeed, for a smooth Fano variety $X$, it is rational homogeneous if and only if $T_X$ is globally generated, or equivalently, the tautological line bundle $\scrO_{\PP T_X}(1)$ is globally generated.
The main result of the present paper is to establish the following weak form of `abundance-type' for the {\hyperref[conj_CP]{Campana-Peternell conjecture}}.  
 
\begin{mainthm}
\label{main-thm}
Let $X$ be a Campana-Peternell manifold. Then tangent bundle $T_X$ is big and semiample in the sense that the tautological line bundle $\scrO_{\PP T_X}(1)$ is big and semiample. 
\end{mainthm}

Note that the semiamplesness of $\scrO_{\PP T_X}(1)$ in the theorem above is just a corollary of the bigness and the base point free theorem \cite[Theorem 3.3, p.75]{KM98}, moreover, the bigness of $T_X$ is equivalent to that of $T_X\oplus\scrO_X$\,, hence it suffices to establish the following:

\begin{mainthm}
\label{mainthm_DPP}
Let $X$ be a Campana-Peternell manifold. Then the tautological line $\scrO_{\PP(T_X\oplus\,\scrO_X)}(1)$ is big on $\PP(T_X\oplus\,\scrO_X)$.
\end{mainthm}

Moreover, let us remark that, to consider $\PP(T_X\oplus\scrO_X)$ rather than $\PP T_X$ has the following two advantages:
\begin{itemize}
\item In the proof of a key result {\hyperref[lemma_sing-B0]{Lemma \ref*{lemma_sing-B0}}}, the argument (the construction of $\scrB$) works only for $\PP T_X\oplus \scrO_X$;
\item In view of the unpublished work \cite{DPP15}, to establish {\hyperref[conj_CP]{Campana-Peternell conjecture}}, it suffices to show that $T_X\oplus \scrO_X$ is strongly semiample (i.e. some symmetric power is globally generated), thus it is indeed more natural to consider $\PP(T_X\oplus\scrO_X)$. 
\end{itemize}
 
The rough idea of the proof is the following: first we apply Viehweg's fibre product trick to prove that the divisor $L$ that equals the sum of the pullbacks of $\scrO(1)$ via projections, is big (and nef), here we should use the ampleness of $-K_X$; then we study the augmented base locus of $L$, indeed, if $\BB_+(L)$ does not contain the diagonal, then by restriction we can easily establish the bigness of $\scrO(1)$; to prove that the diagonal is not contained in $\BB_+(L)$, it suffices to  show that any projection restricted to $\BB_+(L)$ is not dominant; but as a simple observation we see that $\PP(T_X\oplus\,\scrO_X)$ cannot be covered by rational curves that intersect trivially with $\scrO(1)$, and this leads us to show that $\BB_+(L)$ is covered by rational curves that intersect trivially with $L$. To establish this uniruledness result for $\BB_+(L)$, we first show that $\BB_+(L)$ cannot have divisorial components, this is done in {\hyperref[sec_codim]{\S \ref*{sec_codim}}}; then we construct these rational curves by running an appropriate MMP (see {\hyperref[sec_uniruled]{\S \ref*{sec_uniruled}}}).

\paragraph{Acknowledgement}
The author is grateful to his former thesis supervisor Sébastien Boucksom for asking him a question on this problem in 2018, which is the starting point of this work. He would also like to thank his former supervisor Junyan Cao for many helpful discussions. The author wants to express his gratitude to Professor Xiangyu Zhou for pointing out some important references and for explaining to him the geometric meaning of the $L^2$-extension theorems. His gratitude also goes to Professor Shini-ichi Matsumura who has pointed to him some important references and given him some suggestions. The author owes a lot to Professor Andreas H\"oring, who has read the preliminary version of this paper, for his suggestions and inspiring questions. The author is supported by the National Natural Science Foundation of China (Grant No. 12288201 and No. 12301060) and the National Key R\&D Program of China (Grant No. 2021YFA1003100).

%% file: V1/preliminaries.tex
\section{Preliminaries}
\label{sec_preliminaries}
The aim of this section to state some preliminary results that will be used in the proof of our main theorem.

\subsection{Preliminary results from the theory of MMP}
\label{ss_pre_MMP}
In this subsection we collect some results from the minimal model theory. Let us remark that these results are valid (and stated in their original form) for any projective morphism $X\to U$ between quasi-projective normal varieties, while the versions cited here are stated in the form of taking $U=\pt$.  

First we need the following cone theorem for lc pairs:
\begin{thm}[{\cite[Theorem 3.8.1 and Corollary 3.8.2]{BCHM10}}]
\label{thm_lc-cone}
Let $(X,\Delta)$ be a projective lc pair of dimension $n$ with $K_X+\Delta$ being $\RR$-Cartier. Suppose there is a divisor $\Delta_0$ on $X$ such that $(X,\Delta_0)$ is klt.
\begin{itemize}
\item[\rm(a)] If $R$ is an extremal ray of $\NEb(X)$ that is $(K_X+\Delta)$-negative, then there is a rational curve $\Gamma\subseteq X$ such that $R=\RR_{\geqslant 0}\cdot[\Gamma]$ and $0<-(K_X+\Delta)\cdot\Gamma\leqslant 2n$. 
\item[\rm(b)] If we can write $\Delta=A+B$ with $A\geqslant 0$ ample and $B\geqslant 0$, then there are only finitely many $(K_X+\Delta)$-negative extremal rays in $\NEb(X)$. 
\end{itemize}
\end{thm}
The key ingredient that will be used later is the finiteness of the number of extremal rays. 
%The following result is a useful lemma to study a sequence of steps of an MMP:
%\begin{lemma}[{\cite[Lemma 3.10.11]{BCHM10}}]
%\label{lemma_lc-centre-MMP}
%Let $(X,\Delta)$ be a projective $\QQ$-factorial dlt pair and $\phi:X\dashrightarrow Y$ be a sequence of steps of a $(K_X+\Delta)$-MMP. Set $\Delta_Y:=\phi_\ast\Delta$. Then
%\begin{itemize}
%\item[\rm(a)] $\phi$ is an isomorphism at the generic point of every non-klt centre of $K_Y+\Delta_Y$\,. In particular, $(Y,\Delta_Y)$ is dlt.
%\item[\rm(b)] Suppose that we can write $\Delta=S+A+B$ with $S=\lfloor\Delta\rfloor$\,, $B\geqslant 0$, $A\geqslant 0$ big and $\BB_+(A)$ does not contain any non-klt centre of $(X,\Delta)$. Then $\phi_\ast S=\lfloor\Delta_Y\rfloor$\,, $\phi_\ast A$ is big and $\BB_+(\phi_\ast A)$ does not contain any non-klt centre of $(Y,\Delta_Y)$.
%\end{itemize}
%\end{lemma}
%Finally we need the following result on the MMP for lc pairs in \cite{Bir12}:
And we also need the following result on the MMP for lc pairs in \cite{Bir12}:
\begin{thm}[{\cite[Corollary 1.2 and Theorem 1.9]{Bir12}}]
\label{thm_lc-flip}
Let $(X,\Delta)$ be a projective lc pair.
\begin{itemize}
\item[\rm(a)] For any extremal $(K_X+\Delta)$-negative flipping contraction, its $(K_X+\Delta)$-flip exists. 
\item[\rm(b)] Suppose that $(X,\Delta)$ is $\QQ$-factorial dlt, and we can write $\Delta=A+B$ with $A\geqslant 0$ ample and $B\geqslant 0$, then the $(K_X+\Delta)$-MMP with scaling of a sufficiently ample divisor terminates.
\end{itemize}
\end{thm}

\subsection{Rational curves on CP-manifolds}
\label{ss_pre_RC}
In this subsection we review some results concerning the Mori cones and the families of rational curves of CP-manifolds. These results have essentially appeared in \cite[\S 2.3 and \S 3]{MOSCWW15}. 

\begin{prop}
\label{prop_cone-CP-mfd}
Let $X$ be a CP-manifold of Picard number $\rho$. Then $\rank N^1(X)=\rank N_1(X)=\rho$ and we have  
\begin{itemize}
\item[{\rm(a)}] The Mori cone $\NEb(X)$ of $X$ is a simplicial (polyhedral) cone, generated by $\rho$ extremal rays $R_1,\cdots, R_\rho$\,; and for each $i$, $R_i=\RR_{\geqslant 0}\cdot[\Gamma_i]$ for some rational curve $\Gamma_i$. 
\item[{\rm(b)}] Let $J\subseteq I:=\{\,1,\cdots,\rho\,\}$ be any subset, then there is an extremal contraction of the extremal face generated by $\{R_i\}_{i\in J}$, denoted by $\phi_J: X\to Y_J$. Moreover, $Y_J$ is a CP-manifold of Picard number $\rho-\#J$ and $\phi_J$ is a smooth fibration whose fibres are CP-manifolds.
\item[{\rm(c)}] The nef cone and the pseudoeffective cone of $X$ coincide, it is the simplicial cone generated by $[\phi_{I_j}^\ast A_{Y_{I_j}}]$ where $I_j:=I\backslash\{j\}$ and $A_{Y_{I_j}}$ is an ample divisor that generates the Picard group of $Y_{I_j}$ (here we use the notations of (b)). In particular, every pseudoeffective divisor on $X$ is semiample.
\end{itemize}
\end{prop}

\begin{proof}
(a) is just \cite[Corollary 3.2]{MOSCWW15} combined with the Cone Theorem (q.v.~\cite[Theorem 3.7, p.76]{KM98}) for Fano varieties. (b) has been stated in \cite[Theorem 5.2]{DPS94} and the correct proof is given in \cite[Theorem 4.4]{SCW04}, see also \cite[Theorem 3.3]{MOSCWW15}. It remains to prove (c). The equality $\Psef(X)=\Nef(X)$ is shown in \cite[Proposition 3.7]{DPS94}, this cone is simplicial since it is dual to the Mori cone $\NEb(X)$ which is simplicial by (a). Before establishing the second part of (c), let us first clarify that $Y_{I_j}$ has Picard number $\rho(Y_{I_j})=\rho-(\rho-1)=1$ by (b), and thus we can take an ample divisor $A_{Y_{I_j}}$ on $Y_{I_j}$\,, so that $\Pic(Y_{I_j})=\ZZ\cdot\scrO(A_{Y_{I_j}})$\,. Set 
\[
S:=\sum_{j=1}^\rho\RR_{\geqslant0}\cdot[\phi_{I_j}^\ast A_{Y_{I_j}}],
\]
which is a subcone of $\Nef(X)$. Since $\left\{[\phi_{I_j}^\ast A_{Y_{I_j}}]\right\}_{j}$ forms a basis of $N^1(X)_{\RR}:=N^1(X)\otimes_{\ZZ}\RR$\,, then for any nef divisor $D$ on $X$, we can write
\[
D\equiv \sum_{j=1}^\rho a_j\phi^\ast_{I_j}A_{Y_{I_j}}
\] 
and we find $a_j\geqslant 0$ from the nefness condition that $D\cdot C_j\geqslant0$\,. Hence $\Nef(X)=S$.  
\end{proof}

Keep the notations as in {\hyperref[prop_cone-CP-mfd]{Proposition \ref*{prop_cone-CP-mfd}}} and let $\RC^n(X)$ be the family of rational curves on $X$ (q.v.~\cite[\S II.2, pp.103-109]{Kollar96}). For each $i\in\{\,1,\dots,\rho\,\}$, we might assume that $\Gamma_i$ has the minimal $(-K_X)$-degree in $R_i$\,. Let $M_i$ be an irreducible component of $\RC^n(X)$ that contains $\Gamma_i$\,, let $\pi_i: U_i\to M_i$ be the universal family over $M_i$ and $e_i: U_i\to X$ the evaluation morphism. By construction, $U_i\subseteq M_i\times X$. 

\begin{prop}
\label{prop_fam-rc-CP-mfd}
Let everything as above and set $d_i:=-K_X\cdot\Gamma_i$\,. Then
\begin{itemize}
\item[{\rm(a)}] $M_i$ is a smooth projective variety of dimension $\dim X+d_i-3$.
\item[{\rm(b)}] $e_i$ is a smooth fibration and $\pi_i$ is a smooth $\PP^1$-fibration. In particular, for any $x\in X$, $(U_i)_x$ is a smooth projective variety of dimension $d_i-2$.
\item[{\rm(c)}] We have a short exact sequence 
\[
0\to T_{U_i/M_i}\to e_i^\ast T_X\to N_{U_i/M_i\times X}\to 0\,,
\]
here $N_{U_i/M_i\times X}$ denotes the normal bundle of $U_i$ in $M_i\times X$\,. 
\item[\rm(d)] We have a short exact sequence
\[
0\to T_{U_i/X}\to\pi_i^\ast T_{M_i}\to N_{U_i/M_i\times X}\to 0.
\]
For any $x\in X$, $\pi_i|_{(U_i)_x}$ equals to the restriction of the inclusion $U_i\hookrightarrow M_i\times X$ and make $(U_i)_x$ as a subvariety of $M_i$, and the exact sequence above gives rise to an isomorphism $N_{U_i/M_i\times X}|_{(U_i)_x}\simeq N_{(U_i)_x/M_i}$.
\end{itemize}
\end{prop}

\begin{proof}
(a) and (b) are essentially \cite[Proposition 2.10]{MOSCWW15}. We provide here a proof for the convenience of the readers. First, by \cite[II.2.12, p.105]{Kollar96} we see that $\pi_i$ is a smooth $\PP^1$-fibration (cf.~\cite[II.2.5, p.105]{Kollar96}). By the property of smooth morphisms (q.v.~\cite[Proposition (17.7.7), pp.74-75]{EGA4-4}),
%Stacks Project Lemma 29.34.19 (Section 29.34)
the smoothness of $e_i$ follows from \cite[II.2.15, p.111]{Kollar96} and \cite[II.3.5.3, p.115]{Kollar96}; by the same argument we show the smoothness of $U_i$ and of $M_i$. By considering the Stein factorization of $e_i$ we get an étale cover of $X$, which is a fortiori trivial since $X$ is Fano (hence simply connected), consequently $e_i$ has connected fibres, and (b) is thus established. To show the projectivity of $M_i$\,, we take an ample divisor $A_0$ on $X$, and for every $m\in\ZZ_{\geqslant 0}$ set 
\[
A_m:=A_0+m\cdot\phi_i^\ast A_{Y_i}
\]
where $\phi_i: X\to Y_i$ is the contraction of $R_i$ and $A_{Y_i}$ an ample divisor on $Y_i$\,. Then $A_m$ is still ample, and $\Gamma_i$ has minimal $A_m$-degree when $m$ is sufficiently large, hence by \cite[II.2.14, p.110]{Kollar96} $M_i$ is projective. To establish (a) it remains to calculate the dimension of $M_i$\,, this can be done by using \cite[II.1.7.1, p.95]{Kollar96} and \cite[II.2.15, p.111]{Kollar96}. As for (c), it suffices to apply the Snake lemma to the following commutative diagram (noting that $T_{M_i\times X}\simeq \pi_i^\ast T_{M_i}\oplus e_i^\ast T_X$):
\begin{center}
\begin{tikzpicture}[scale=2.0]
\node (A0) at (-0.8,0) {$0$};
\node (A1) at (0,0) {$T_{U_i}$};
\node (A2) at (1.5,0) {$\pi_i^\ast T_{M_i}\oplus e_i^\ast T_X$};
\node (A3) at (3,0) {$N_{U_i/M_i\times X}$};
\node (A) at (4,0) {$0$};
\node (B0) at (-0.8,-1) {$0$};
\node (B1) at (0,-1) {$\pi_i^\ast T_{M_i}$};
\node (B2) at (1.5,-1) {$\pi_i^\ast T_{M_i}$};
\node (B3) at (3,-1) {$0$,};
\path[->,font=\scriptsize,>=angle 90]
(A0) edge (A1)
(A1) edge (A2)
(A2) edge (A3)
(A3) edge (A)
(B0) edge (B1)
(B1) edge (B2)
(B2) edge (B3)
(A1) edge (B1)
(A2) edge (B2)
(A3) edge (B3);
\end{tikzpicture}  
\end{center}
where the rows of the diagram are exact. (d) can be proved by the same argument as for (c).
\end{proof}

\subsection{Numerically flat vector bundles}
\label{ss_pre_num-flat}
Throughout this subsection, let $X$ be a compact K\"ahler manifold or projective variety and $E$ a vector bundle on $X$. Recall that $E$ is called \emph{numerically flat} if both $E$ and $E^\ast$ are nef, or equivalently, if $E$ is nef and $c_1(E)=0$ (q.v.~\cite[Definition 1.17]{DPS94}). Clearly if $E$ is Hermitian flat, then $E$ is numerically flat; conversely, a numerically flat vector bundle is successive extension of Hermitian flat vector bundles (\cite[Theorem 1.18]{DPS94}). The following simple result is of course known to experts, but due to lack of reference, we provide a proof here for the convenience of the readers.
\begin{lemma}
\label{lemma_h0-num-flat}
Let $E$ be a numerically flat vector bundle of rank $r$ on $X$. Then $\dimcoh^0(X,E)\leqslant r$.  
\end{lemma}
\begin{proof}

By induction on $r$. The proof is similar to that of \cite[Lemma 4.3.3]{Cao13}. The line bundle case follows from \cite[Proposition 1.16]{DPS94}. For the general case, if $\dimcoh^0(X,E)=0$, then there is nothing to prove. Otherwise, take a non-zero section $\sigma\in\Coh^0(X,E)$ which induces an inclusion $\scrO_X\hookrightarrow E$ which we still denote by $\sigma$, since $E^\ast$ is nef, by \cite[Proposition 1.16 and Theorem 1.20]{DPS94} we see that $\scrO_X$ is a subbundle of $E$. Set $E':=\Coker(\sigma)$, then $E'$ is a numerically flat vector bundle of rank $r-1$, and our lemma follows from the induction hypothesis. 
\end{proof}

See \cite[Proposition 2.8]{CCM21} for a related result.

%And we observe the following simple corollary of the above lemma:
%\begin{cor}
%\label{cor_Pn-h0-num-flat}
%Let $E$ be a numerically flat vector bundle of rank $r$ on $\PP^n$, then 
%\[
%\dimcoh^0(\PP^n,E\otimes\scrO(1))\leqslant (n+1)\cdot r.
%\]
%\end{cor}
%\begin{proof}
%By induction on $n$. Set $E(1):=E\otimes\scrO(1)$. If $n=1$, then $E$ is trivial and we have $\dimcoh^0(\PP^1,E(1))=2r$, in particular the inequality holds. 

%For the general case, let $H\in\left|\,\scrO(1)\,\right|$ be a general hyperplane, then $H\simeq\PP^{n-1}$ and we have the short exact sequence
%\[
%0\to E\to E(1)\to E(1)|_H\to 0.
%\]
%Then by the induction hypothesis and {\hyperref[lemma_h0-num-flat]{Lemma \ref*{lemma_h0-num-flat}}} we get
%\begin{align*}
%\dimcoh^0(\PP^n,E(1)) &\leqslant \dimcoh^0(\PP^n,E)+\dimcoh^0(H,E(1)|_H)\\
%&\leqslant r+n\cdot r=(n+1)\cdot r.
%\end{align*}
%\end{proof}

\subsection{Asymptotic order of vanishing}
\label{ss_pre_AOV}
In this subsection we recall the definition and basic properties of the asymptotic order of vanishing of a Cartier divisor along a subvariety. 

Let $X$ be a normal variety and $Z$ an irreducible subvariety of $X$ that is not contained in $X_{\sing}$. For an effective Cartier divisor $D$ on $X$, the \emph{multiplicity of $D$ along $Z$}, denoted by $\mult_Z(D)$, is defined to be the order of vanishing of $D$ at a general point of $Z$, or equivalently, $v(D)$ where $v$ is the (unique) divisorial valuation whose centre is $Z$ (q.v.~\cite[Definition 2.9]{ELMNP06}); and it is equal to the generic Lelong number of the (closed positive) current $[D]$ along $Z$ (q.v.~\cite[(2.8) and (2.17), pp.19-23]{Dem10}). 

In the sequel, suppose that $X$ is projective. Given a linear system $|V|\subseteq|D|$, the multiplicity of $|V|$ along $Z$, denoted by $\mult_Z|V|$, is defined to be the multiplicity along $Z$ of a general divisor in $|V|$ (cf.~\cite[Vol.I, \S 2.3.C, Definition 2.3.11, p.165]{Laz04}). Similarly, it is equal to the generic Lelong number of the curvature current of the Hermitian metric on $\scrO_X(D)$ defined by $|V|$. 

Now suppose that $D$ is a big ($\QQ$-Cartier) $\QQ$-divisor, then the asymptotic order of vanishing of $D$ along $Z$ is defined to be
\[
\sigma_Z(D):=\lim_{m\to+\infty}\frac{\mult_Z(|mD|)}{m}.
\]
Here the limit exists by \cite[\S III.1.a, pp.79-81]{Nak04} or \cite[Definition 2.2]{ELMNP06}, cf. \cite[\S 2.2(1)]{Tak08}. And one can extend $\sigma_Z$ to a function defined on the whole of $\Psef(X)$ (q.v.~\cite[\S III.1.a, pp.79-85]{Nak04}), and even for analytic pseudoeffective $(1,1)$-classes (q.v.~\cite{Bou04}). The asymptotic order of vanishing has the following properties (q.v.~\cite[Lemma 2.4]{Tak08}):

\begin{prop}
\label{prop_AOV}
Let $X$ be a normal projective variety and $Z$ an irreducible subvariety of $X$ that is not contained in $X_{\sing}$. For pseudoeffective divisors $D$ and $D'$ on $X$, we have
\begin{itemize}
\item[(a)] $\sigma_Z(mD)=m\cdot\sigma_Z(D)$ for any $m\in\ZZ_{>0}$, i.e. $\sigma_Z$ is homogeneous;
\item[(b)] $\sigma_Z(D+D')\leqslant\sigma_Z(D)+\sigma_Z(D')$, i.e. $\sigma_Z$ is subadditive;
\item[(c)] $Z$ is contained in $\BB_-(D)$ if and only if $\sigma_Z(D)>0$.  
\end{itemize}
\end{prop}

\subsection{Multiplier ideals and positivity of direct images}
\label{ss_pre_multiplier-ideals}
In this subsection let us recall the basic definition and some results concerning the multiplier ideals.Throughout this subsection let $X$ be a complex manifold and $L$ a holomorphic line bundle over $X$. For a singular Hermitian metric $h$ on $L$ with semipositive curvature current, its multiplier ideal $\scrJ(h)\subseteq\scrO_X$ is defined as the ideal sheaf generated by germs of holomorphic functions $f$ such that $|f|^2e^{-2\phi}$ is locally integrable where $\phi$ is the local weight of $h$ (q.v.~\cite[Definition (5.4), p.40]{Dem10}). In particular, if $h$ is induced by a linear series $|V|\subseteq|L|$, then we simply write $\scrJ(|V|)$ for $\scrJ(h)$, cf.~\cite[Vol.II, \S 9.3.D, pp.176-177]{Laz04}. 

Multiplier ideals are related to geometric problems via the following result due to Henri Skoda (q.v.~\cite[Lemma (5.6), p.40]{Dem10} and \cite[Vol.II, \S 9.3.A, pp.162-164]{Laz04}):
\begin{prop}
\label{prop_multiplier-multiplicity}
Let $Z$ be a closed subvariety of $X$ of codimension $c$, if the Lelong number of $\phi$ along $Z$ is greater or equal to $c+s-1$ for some $s\in\ZZ_{>0}$, then $\scrJ(h)\subseteq\scrI_Z^{\langle s\rangle}$.
\end{prop}
In the proposition above, the \emph{$s$-th symbolic power of $\scrI_Z$} is defined as the ideal sheaf consisting of germs of holomorphic functions whose vanishing order along $Z$ is greater or equal to $s$ (q.v.~\cite[Vol.II, \S 9.3.A, Definition 9.3.4, p.164]{Laz04}). In addition, if $h$ is induced by a linear series $|V|\subseteq|L|$, then the condition on the Lelong number can translated into the condition on the multiplicity $\mult_Z|V|\geqslant c+s-1$.

Multiplier ideals are powerful tools in the study of complex geometry and algebraic geometry, thanks to the powerful vanishing theorems of Nadel and $L^2$-extension theorems of Ohsawa-Takegoshi associated to them. Since these topics are already treated abundantly in the literatures, 
%\cite{Dem01,Dem10,Laz04}
we will not repeat them here, except the following one, 
%due to the lack of references.
which is a special case of \cite[Theorem 1.7]{Mat22} (see also \cite[Theorem 1.2]{Mat18}) and will be used in the proof of our main theorem. 

\begin{prop}[Nadel vanishing in the relative setting for a linear series]
\label{prop_rel-Nadel-vanishing}
Let $X$ be projective and $f:X\to Y$ a surjective morphism. Let $L$ and $D$ be integral divisors on $X$ and $c\in\QQ_{>0}$ such that $L-c\cdot D$ is big and nef for $f$. Then for any linear series $|V|\subseteq|D|$ and every $i>0$ we have
\[
\RDer^i\!f_\ast\left(\scrO_X(K_X+L)\otimes\scrJ(c\cdot|V|)\right)=0
\]
\end{prop}

Moreover, multiplier ideals appear in the statement of the positivity results on the direct images of twisted relative canonical bundles. Such positivity results originate from Philip A. Griffiths's work \cite{Grif70} on the semipositivity of the Hodge metric. In order to treat more general algebro-geometric problem (such as the Iitaka conjecture $C_{n,m}$), it is natural to consider the $L^2$-metric on the direct image of the relative canonical bundle twisted by a line bundle $L$ equipped with a singular Hermitian metric $h$. If the fibration is smooth, one deduces the positivity result follows from the log-plurisubharmonicity of the Bergman kernel function (similar ideas has implicitly appear in \cite[Lemma 7]{Kaw82}, see also \cite[Lemma 3.2]{Tsu07}); this is done in Berndtsson's work \cite{Ber09} for smooth $h$, and in Berndtsson-P{\u a}un's work \cite{BP08} for general singular $h$. See P{\u a}un's survey \cite{Pau16} for the definition of Griffiths (semi)positivity for singular metrics on vector bundles and torsion free sheaves. For non-necessarily smooth fibrations, P{\u a}un-Takayama use the semistable reduction techniques to obtain a general positivity result (q.v.~\cite{PT18}). However, these results have the drawback that essentially they cannot treat the direct image of the relative canonical bundle twisted by $\scrJ(h)$; note that if $h$ is singular this direct image is the more natural object to consider. In order to overcome this difficulty,  Guan-Zhou observe in \cite[\S 3.5, pp.~1149-1151]{GZ15} and their successive work \cite[Proposition 1.10]{GZ17} that the log-plurisubharmonicity of the Bergman kernel function, as well as the positivity of the direct image, can be deduced from the $L^2$-extension theorem with optimal estimate; the crucial point here is the optimal constant and this deduction can also be regarded as the geometric meaning of the optimal extension. Guan-Zhou's idea is carried out in the successive works \cite{HPS18,DWZZ24,Wang21} to obtain more general positivity results. By these works we have the following theorem:
\begin{thm}[{\cite[Theorem 2.6]{Wang21}}]
\label{thm_pos-direct-image}
Let $f:X\to Y$ be a proper morphism between K\"ahler manifolds and let $(L,h)$ be a line bundle on $X$ equipped with a singular Hermitian metric with semipositive curvature current. Then the $L^2$-metric on $f_\ast(K_X\otimes L\otimes\scrJ(h))$ is Griffiths semipositive and $L^2$-reflexive.
\end{thm}
See \cite[Definition 1.12]{Wang21} for the definition of \emph{$L^2$-reflexivity}; briefly speaking, this means that a local section defined out of an analytic subspace $Z$ of codimension $\geqslant 2$ that is $L^2$-integrable with respect to $h$ can be extended across $Z$. Be careful that in the cited paper this is called `$L^2$-extension property', but in order to avoid possible confusion with \cite[Definition 20.1]{HPS18} we change the terminology here. Of course a Griffiths semipositive singular Hermitian metric on a reflexive sheaf is automatically $L^2$-reflexive. The $L^2$-reflexivity is crucial in the following result:
\begin{prop}[{\cite[Theorem 1.13]{Wang21}}]
\label{prop_pos+det=0}
Let $(\scrF,h)$ be a torsion free sheaf equipped with a singular Hermitian metric that is Griffiths semipositive, if $(\scrF,h)$ is $L^2$-reflexive and $c_1(\det\scrF)=0$, then $(\scrF,h)$ is a Hermitian flat vector bundle. 
\end{prop}
Note that without the $L^2$-reflexivity condition the proposition above does not hold, e.g. an ideal sheaf of a subvariety of codimension $\geqslant 2$ equipped with the trivial metric (q.v.~\cite[Remark after Theorem 1.13]{Wang21}). See \cite[Theorem 3.4]{HIM22} and \cite[Corollary]{Wu22} for related results.

%% file: V1/setup.tex
\section{General setup}
\label{sec_setup}
From now on, let $X$ be a CP-manifold, that is, a smooth Fano variety with nef tangent bundle $T_X$. 
%Suppose that the Picard number of $X$ equals $1$. 
By virtue of \cite[Theorem 3.1]{CP91} and \cite[Theorem 1.2]{OSCWW17}, we make the following 
\begin{ass}
\label{assumption_dim-FT}
$X$ has dimension $\geqslant 3$ and $X$ is NOT of flag type.
\end{ass}
Here a smooth Fano variety is called \emph{of flag type} or a \emph{flag-type manifold} if every elementary contraction of it is a $\PP^1$-fibration (q.v~\cite[Definition 1]{MOSCW15}). 

Put $V:=\PP E$ with $E:=T_X\oplus\scrO_X$ and for every $r\in\ZZ_{>0}$ put
\[
V_r:=\underbrace{V\underset{X}{\times}\cdots\underset{X}{\times}V}_{r\text{ times}},
\]
and let $q_r:V_r\to X$ be the natural morphism. Set $W:=V_{n+1}$ and $q:=q_{n+1}$, where $n:=\dim X$. Then we have $\dim V=2n$ and $\dim W=n^2+2n$. And let $p:V\to X$ be the natural morphism and $\pr_i:W\to V$ be the $i$-th projection, thus we have the following commutative diagram: 
\begin{center}
\begin{tikzpicture}[scale=2.5]
\node (A) at (0,0) {$X$.};
\node (B) at (0,1) {$V:=\PP(T_X\oplus\scrO_X)$};
\node (C) at (-2,0.8) {$W:=\underbrace{V\underset{X}{\times}\cdots\underset{X}{\times}V}_{(n+1)\text{ times}}$};
\path[->,font=\scriptsize,>=angle 90]
(-1.4,1) edge node[above]{$\pr_i$} (B)
(C) edge node[below left]{$q$} (A)
(B) edge node[right]{$p$} (A);
\end{tikzpicture}  
\end{center}
Notice that here we have $q=p\circ\pr_i$ for very $i=1,\cdots,n+1$. Moreover, for every $i\in\{1,\cdots,n+1\}$ we have the following Cartesian diagram
\begin{center}
\begin{tikzpicture}[scale=2.5]
\node (A) at (0,0) {$X$.};
\node (B) at (0,1) {$V$};
\node (C) at (-1,0) {$V_n$};
\node (D) at (-1,1) {$W$};
\node (S) at (-0.5,0.5) {$\square$};
\path[->,font=\scriptsize,>=angle 90]
(D) edge node[above]{$\pr_i$} (B)
(D) edge node[left]{$\pi_i$} (C)
(B) edge node[right]{$p$} (A)
(C) edge node[below]{$q_n$} (A);
\end{tikzpicture}  
\end{center}

Let $H$ be a tautological divisor on $V=\PP E$, i.e. so that $\scrO_V(H)\simeq\scrO_{\PP E}(1)$, and set $P:=-q^\ast K_X$. Since $T_X$ is nef and $-K_X$ is ample, $H$ is nef and $P$ is semiample. Put 
\[
L:=\sum_{i=1}^{n+1}\pr_i^\ast H,
\]
and for every $j\in\{1,2,\cdots,n+1\}$ put
\[
L_j:=\sum_{i\neq j}\pr_i^\ast H.
\]
Then $L$ and the $L_j$'s are nef divisors on $W$ and $L_j$ is a pullback of a divisor on $V_n$ via $\pi_j$\,.

\subsection{Formul{\ae} for canonical divisors} 
In this subsection we will calculate the canonical divisors of $V$ and $W$. The computation is standard. From the relative Euler sequence we find $K_{V/X}\sim -(n+1)H-p^\ast K_X$ and by a recursive calculation we get
\[
K_{W/X}\sim\sum_{i=1}^{n+1} \pr_i^\ast(-(n+1)H-p^\ast K_X)\sim -(n+1)L+(n+1)P. 
\]
In consequence we have 
\begin{equation}
\label{formulae_canonical}
K_V\sim -(n+1)H \quad\text{ and }\quad K_W\sim -(n+1)L+nP.
\end{equation}

\subsection{Viehweg's fibre product trick}
In this subsection we will prove the following positivity results:
\begin{prop}
\label{prop_L+-tP}
Let everything as above. Then
\begin{itemize}
\item[(a)] $L+tP$ is an ample $\QQ$-divisor for every $t\in\QQ_{>0}$\,; 
\item[(b)] $L_j-P\geqslant 0$ for every $j$ and $L-P\geqslant 0$;
\item[(c)] $L-tP$ is a big $\QQ$-divisor for every rational number $t<1+1/n$.
\end{itemize}
\end{prop}

\begin{proof}
First, a recursive computation gives the following identity for the direct images:
\begin{equation}
\label{formula_direct-im}
q_\ast\scrO_W(mL)\simeq (\Sym^mE)^{\otimes (n+1)}\quad\text{ and }\quad q_\ast\scrO_W(mL_j)\simeq (\Sym^mE)^{\otimes n}
\end{equation}
for every $m\in\ZZ_{>0}$.

To prove (a), just note that $W$ is can be regarded as a subvariety of $\PP(E^{\otimes(n+1)})$ via the relative Segre embedding, thus (a) follows from the observation that $E^{\otimes(n+1)}\langle-tK_X\rangle$ is an ample $\QQ$-twisted vector bundle (q.v.~\cite[Vol.II, \S 6.2.A, pp.~20-21 \& \S 6.2.B, Proposition 6.2.11, p.~24]{Laz04}) for every $t\in\QQ_{>0}$. 

As for (b), first note that the natural inclusion $\scrO_X(-K_X)\hookrightarrow T_X^{\otimes n}$ induces a non-zero section in $\Coh^0(X,T_X^{\otimes n}(K_X))$, but $T_X^{\otimes n}$ is direct summand of $E^{\otimes n}$, hence we get a non-zero section in
\[
\Coh^0(X,E^{\otimes n}(K_X))\simeq\Coh^0(W,\scrO_W(L_j-P)),
\]
here the isomorphism follows from the equality \eqref{formula_direct-im}. Similarly from the natural inclusion $\scrO_X(-K_X)\simeq\det E\hookrightarrow E^{\otimes (n+1)}$ we see that $\Coh^0(W,\scrO_W(L-P))\neq 0$.

Finally, (c) follows from (a) and Kodaira's Lemma (q.v.~\cite[Vol.I, \S 2.2.A, Corollary 2.2.7, p.~141]{Laz04}) by noting that 
\[
nL-(n+1)P=\sum_{j=1}^{n+1}L_j-P\geqslant 0.
\]
%By writing 
%\[
%L-tP\sim \frac{1+t}{2}(L-P)+\frac{1-t}{2}(L+P)
%\] 
%we see that (b) follows from Kodaira's Lemma (q.v.~\cite[Vol.I, \S 2.2.A, Corollary 2.2.7, p.~141]{Laz04}). 
\end{proof}

%\begin{lemma}[{\cite[Lemma 4.9]{Cam04}}]
%\label{lemma_kod-eff+pull-ample}
%Let $f:X\to Y$ be a fibre space with $X$ a normal complex variety and $Y$ a projective variety. Let $L$ be a holomorphic line bundles on $X$ such that $\kappa(L)\geqslant 0$ and let $A$ be a ample line bundle on $Y$. Then 
%\[
%\kappa(X,L\otimes f^\ast A)=\kappa(F,L|_F)+\dim Y
%\]
%where $F$ denotes the general fibre of $f$.
%\end{lemma}

%Here $\kappa$ denotes the Kodaira-Iitaka dimension of a ($\QQ$-)line bundle or Cartier divisor. let us remark that this simple but useful result has been implicitly used in the works on the Iitaka conjecture $C_{n,m}$\,, e.g. \cite{Esn81, Vie82, Vie83}; it is explicitly formulated in \cite[Lemma 4.9]{Cam04} but without proof. See \cite[Lemma 3.1]{Wang21} for an explicit proof.

%% file: V1/codim_B+.tex
\section{Codimension estimate of the augmented base locus of \texorpdfstring{$L$}{text}}
\label{sec_codim}
%From now on let us suppose that the Picard number of $X$ equals $1$ (hence $\Pic(X)\simeq\ZZ$) and let $H_X$ be the ample generator of $\Pic(X)$ (so that $-K_X$ is a positive multiple of $H_X$). 
Let everything as in the {\hyperref[sec_setup]{\S \ref*{sec_setup}}}. From now on, we will concentrate on the study of the augmented base locus $\BB_+(L)$ of $L$.
Since $L+P$ is ample by {\hyperref[prop_L+-tP]{Proposition \ref*{prop_L+-tP}(a)}}, we have 
\[
\BB_+(L)=\BB(L-\epsilon P)=\Bs|\,k(L-\epsilon P)\,|_{\red}
\]
for some $\epsilon\in\QQ_{>0}$ sufficiently small and for any $k\in\ZZ_{>0}$ sufficiently large and divisible (q.v.~\cite[Vol.II, \S 10.3, Lemma 10.3.1, p.~247; Vol.I, \S 2.1.A, Proposition 2.1.21, p.~127]{Laz04}). In this section we prove:
\begin{prop}
\label{prop_codim-B+}
$\BB_+(L)$ has codimension at least $2$ in $W$.
\end{prop}

\begin{proof} 
Assume that $\codim\BB_+(L)=1$, and let $D$ be the divisorial part of $\BB_+(L)$, in other word, $D=\left(\Fix\left|\,k(L-\epsilon P)\,\right|\right)_{\red}$\,,
and $D$ can be regarded as a reduced subscheme of $W$. 
%By construction $\dimcoh^0(W,\scrO_W(mD))=1$ for any $m\in\ZZ_{>0}$ sufficiently divisible.

\paragraph{Step 1.} We show that $D\sim L-P$.

%And since $\BB_+(L)$ must be invariant under the natural $\mathfrak{S}_n$-action on $W$, we can write $D=aL-bq^\ast H_X$. First note that we must have $a>0$: indeed, if $a=0$, then a fortiori $b<0$ and $kL-(k\epsilon-b)q^\ast H_X$ is movable, thus $\codim\BB_+(L)\geqslant 2$, contradicting our assumption; it remains to exclude the case $a<0$, indeed, restrict $D$ to a general fibre $W_x\simeq(\PP^{n-1})^{\times n}$ of $q$, we get a nonzero section of $\scrO_{(\PP^{n-1})^{\times n}}(a,\cdots,a)$, which is impossible if $a<0$. 

%As a consequence, we must have $b>0$. 

%For $m\in\ZZ_{>0}$ sufficiently large and divisible integer. Then for each $i=0,\cdots,m-1$, consider the short exact sequence
%\[
%0\to \scrO_W(amL-(i+1)D)\to \scrO_W(amL-iD)\to \scrO_D((amL-iD)|_D)
%\]

To this end, consider the short exact sequence:
\begin{equation}
\label{ses_L-P-D}
0\to \scrO_W(L-P-D)\to\scrO_W(L-P)\to\scrO_D(L-P)\to 0,
\end{equation}
which gives rise to the following exact sequence
\[
0\to\Coh^0(W,\scrO_W(L-P-D))\xrightarrow{\phi}\Coh^0(W,\scrO_W(L-P))\xrightarrow{\text{restr.}}\Coh^0(D,\scrO_D(L-P)).
\]
But by construction $D\subseteq\left(\Bs|k(L-\epsilon L)|\right)_{\red}\subseteq\Bs|L-P|$\,, hence the restriction morphism 
\[
\Coh^0(W,\scrO_W(L-P))\xrightarrow{\text{restr.}}\Coh^0(D,\scrO_D(L-P))
\]
is the zero map, and thus the inclusion morphism $\phi$ is an isomorphism. In particular, 
%since 
%\[
%\Coh^0(W,\scrO_W(L-P))\simeq\Coh^0(X,T_X^{\otimes n}(K_X))\neq 0\,,
%\]
by {\hyperref[prop_L+-tP]{Proposition \ref*{prop_L+-tP}(b)}} we see that $L-P-D\geqslant 0$. Since $\BB_+(L)$ is invariant under the natural $\mathfrak{S}_n$-action on $W$, and so is $D$, we can write $D\sim aL-q^\ast D_X$ with $a\in\ZZ$ and $D_X$ a divisor on $X$.

First note that $0\leqslant a\leqslant 1$, indeed this can be shown by looking at the restriction of $L-P-D$ and $D$ to the general fibre of $q$. 
If $a=0$, then $D\sim -q^\ast D_X$, hence $-D_X\geqslant 0$, and hence $-D_X$ is semiample by {\hyperref[prop_cone-CP-mfd]{Proposition \ref*{prop_cone-CP-mfd}(c)}} and so is $D$, which is impossible since $D=\left(\Fix\left|\,k(L-\epsilon P)\,\right|\right)_{\red}$\,. 
Hence we must have $a=1$, and consequently $K_X+D_X\geqslant 0$. We can rewrite the isomorphism $\phi$ as following: 
\[
\Coh^0(X,\scrO_X(K_X+D_X))\xrightarrow{\simeq}\Coh^0(X,E^{\otimes(n+1)}(K_X)).
\]
%On the other hand, apply $q_\ast$ to the short exact sequence \eqref{ses_L-P-D} we get an inclusion $\scrO_X(K_X+D_X)\hookrightarrow T_X^{\otimes n}(K_X)$. Since $\scrO_X$ is a direct summand of $T_X^{\otimes n}(K_X)$, we can consider the composite morphism 
%\[
%\scrO_X(K_X+D_X)\hookrightarrow T_X^{\otimes n}(K_X)\twoheadrightarrow\scrO_X\,,
%\]
On the other hand, apply $q_\ast$ to the short exact sequence \eqref{ses_L-P-D} we get an inclusion $\scrO_X(K_X+D_X)\hookrightarrow E^{\otimes (n+1)}(K_X)$. Since $\scrO_X$ is a direct summand of $T_X^{\otimes n}(K_X)$ thus of $E^{\otimes (n+1)}(K_X)$, we can consider the composite morphism 
\[
\scrO_X(K_X+D_X)\hookrightarrow E^{\otimes (n+1)}(K_X)\twoheadrightarrow\scrO_X\,,
\]
denoted by $\psi$, which gives rise to the following morphism between cohomology groups
\[
\Coh^0(X,\scrO_X(K_X+D_X))\xrightarrow[\simeq]{\phi}\Coh^0(X,E_X^{\otimes(n+1)}(K_X))\twoheadrightarrow\Coh^0(X,\scrO_X).
\]
The morphism above is non-zero, which implies that $\psi:\scrO_X(K_X+D_X)\to\scrO_X$ is a non-zero morphism, and thus $K_X+D_X\leqslant 0$. But we have seen that $K_X+D_X\geqslant 0$, hence $D_X\sim -K_X$ and $D\sim L-P$.

\paragraph{Step 2.} We show that $s_k(\Omega_X)=0$ for every $k>1$, here $s_k$ represents the $k$-th Segre class. 

We follow the definitions and notations in \cite[\S 3.1-3.2, pp.~47-51]{Ful84} for the Segre classes of a vector bundle. Note that the definition of Segre classes in \cite[\S 3.1, p.~47]{Ful84} is different from that in \cite[\S 2, p.~21]{DPS94} (cf. \cite[Vol.II, \S 8.3, Example 8.3.5, p.~118]{Laz04}) 

By Nakamaye's theorem on base loci \cite[Theorem 0.3]{Nak00} (see also \cite[\S 10.3, Theorem 10.3.5, p.~248]{Laz04} and \cite[Theorem 1.4]{Bir17}), we have 
\[
L^{n^2+2n-1}\cdot D=L^{n^2+2n-1}\cdot (L-P)=0.
\]
Since $L=\sum_{i=1}^{n+1}\pr_i^\ast H$ and since $H$ is nef, we have 
\begin{equation}
\label{eq_Nakamaye}
(\pr_1^\ast H)^{k_1}\cdot(\pr_2^\ast H)^{k_2}\cdot\cdots\cdot(\pr_{n+1}^\ast H)^{k_{n+1}}\cdot D=0,
\end{equation}
for any $k_1,\cdots, k_{n+1}\in\{0,1,\cdots, 2n-1\}$ satisfying $k_1+\cdots+k_{n+1}=n^2+2n-1$. 

From the definition of Segre classes and the fact that $s_k(E^\ast)=s_k(\Omega_X)$ (q.v.~\cite[Theorem 3.2(e)]{Ful84}), we have the following formul\ae{}:
\begin{align*}
(\pr_1^\ast H)^{n+a_1}\cdot(\pr_2^\ast H)^{n+a_2}\cdot\cdots\cdot(\pr_{n+1}^\ast H)^{n+a_{n+1}} &= s_{a_1}(\Omega_X)\cdot s_{a_2}(\Omega_X)\cdot\cdots\cdot s_{a_{n+1}}(\Omega_X), \\
(\pr_1^\ast H)^{n+b_1}\cdot(\pr_2^\ast H)^{n+b_2}\cdot\cdots\cdot(\pr_{n+1}^\ast H)^{n+b_{n+1}}\cdot P &= s_{b_1}(\Omega_X)\cdot s_{b_2}(\Omega_X)\cdot\cdots\cdot s_{b_{n+1}}(\Omega_X)\cdot(-K_X),
\end{align*}
where $a_1,\cdots,a_n\in\{0,1,\cdots, n\}$ satisfying $a_1+\cdots+a_n=n$ and $b_1,\cdots, b_n\in\{0,1,\cdots, n-1\}$ satisfying $b_1+\cdots+b_n=n-1$. For every $k\in\{1,\cdots, n-1\}$, set
\[
\gamma_k:=(\pr_1^\ast H)^{n+k}\cdot(\pr_2^\ast H)^{n+1}\cdot\cdots\cdot(\pr_{n-k}^\ast H)^{n+1}\cdot(\pr_{n-k+1}^\ast H)^n\cdot\cdots\cdot(\pr_n^\ast H)^n,
\]
regarded as an element in $\Chow^{n^2+2n-1}(W)$, then by the formula above we have
\begin{equation*}
\gamma_k\cdot(\pr_i^\ast H) = \left\{
\begin{array}{rl}
s_{k+1}(\Omega_X)\cdot(-K_X)^{n-k-1} &\text{if } i=1, \\
s_k(\Omega_X)\cdot s_2(\Omega_X)\cdot (-K_X)^{n-k-2} &\text{if } 2\leqslant i\leqslant n-k, \\
s_k(\Omega_X)\cdot (-K_X)^{n-k} &\text{if } n-k+1\leqslant i\leqslant n,
\end{array}
\right.
\end{equation*}
and 
\[
\gamma_k\cdot P=s_k(\Omega_X)\cdot(-K_X)^{n-k}.
\]
Combine the above formul\ae{} with \eqref{eq_Nakamaye} we have 
\begin{align*}
0 & =\gamma_k\cdot D=\gamma_k\cdot (\pr_1^\ast H+\cdots+\pr_n^\ast H-P) \\
& =s_{k+1}(\Omega_X)\cdot(-K_X)^{n-k-1}+(n-k-1)s_k(\Omega_X)\cdot s_2(\Omega_X)\cdot(-K_X)^{n-k-2} +(k-1)s_k(\Omega_X)\cdot(-K_X)^{n-k}.
\end{align*}
By \cite[Theorem 2.5 and the discussion thereafter]{DPS94}, every term on the last row of the above equation is non-negative, hence is vanishing, in particular we have 
\[
s_{k+1}(\Omega_X)\cdot(-K_X)^{n-k-1}=0
\]
for every $k\geqslant 1$. In other word, for every $k\in\{2,\cdots, n\}$, $s_k(\Omega_X)\cdot(-K_X)^{n-k}=0$. To conclude, it suffices to show that for any (reduced and irreducible) subvariety $Z$ of $X$ of dimension $k$, we have $s_k(\Omega_X)\cdot Z=0$. Since $-K_X$ is ample, there is $m\in\ZZ_{>0}$ such that $\scrO_X(-mK_X)\otimes\scrI_Z$ is globally generated, then $Z$ is a component of the complete intersection of $(n-k)$ general members of the linear system defined by $\Coh^0(X,\scrO_X(-mK_X)\otimes\scrI_Z)$, then by \cite[Theorem 2.5]{DPS94} we have $s_k(\Omega_X)\cdot Z=0$.

To end up this step, we will show that, as a consequence of the vanishing of the higher Segre classes, we have
\begin{equation}
\label{formula_Chern_TX}
c_k(T_X)=(-K_X)^k
\end{equation}
for every $k\in\ZZ_{>0}$. Indeed, by definition (q.v.~\cite[\S 3.2, p.~50]{Ful84}), for every $k\in\ZZ_{>0}$ we have 
\[
c_k(\Omega_X)+c_{k-1}(\Omega_X)\cdot s_1(\Omega_X)+\cdots+c_1(\Omega_X)\cdot s_{k-1}(\Omega_X)+s_k(\Omega_X)=0,
\]
thus 
\[
c_k(\Omega_X)=-s_1(\Omega_X)\cdot c_{k-1}(\Omega_X)=K_X\cdot c_{k-1}(\Omega_X). 
\]
Using the Chern class formula for dual bundles (q.v.~\cite[Remark 3.2.3(a), p.~54]{Ful84}) to rewrite the above equation we get
\[
c_k(T_X)=-K_X\cdot c_{k-1}(T_X),
\]
and the desired formula \eqref{formula_Chern_TX} follows.

\paragraph{Step 3.} We prove that $X$ is of flag type and conclude.  

%Keep the notations as in {\hyperref[sec_preliminaries]{\S \ref*{sec_preliminaries}}}, for simplicity we omit the subscript $i$ 
Let everything as in {\hyperref[prop_fam-rc-CP-mfd]{Proposition \ref*{prop_fam-rc-CP-mfd}}}, we will show that the contraction of the extremal ray $R_i$ is a $\PP^1$-fibration for every $i$. In the sequel, for simplicity we omit the subscript $i$, and the situation is summarized in the following diagram
\begin{center}
\begin{tikzpicture}[scale=2.5]
\node (A) at (0,0) {$M$};
\node (B) at (0,1) {$U\subseteq M\times X$};
\node (C) at (1.2,1) {$X$.};
\path[->,font=\scriptsize,>=angle 90]
(B) edge node[left]{$\pi$} (A)
(B) edge node[above]{$e$} (C);
\end{tikzpicture}  
\end{center}

Set 
\[
\alpha:=-\frac{1}{d}e^\ast K_X,
\] 
then $\beta:=2\alpha+K_{U/M}$ is numerically trivial on each fibre of $\pi$, and since $\pi$ is a Fano fibration, by \cite[Theorem 3.7(4), p.76]{KM98} $\beta\sim_{\QQ}\pi^\ast\beta_M$ for some divisor $\beta_M$ on $M$. From the short exact sequence {\hyperref[prop_fam-rc-CP-mfd]{Proposition \ref*{prop_fam-rc-CP-mfd}(c)}}
\[
0\to T_{U/M}\to e^\ast T_X\to N_{U/M\times X}\to 0
\]
we find $\det N_{U/M\times X}\sim_{\QQ}(d-2)\alpha+\beta$; moreover, by \cite[Theorem 3.2(e)]{Ful84}, for any $k\in\ZZ_{>0}$ we have
\[
c_k(e^\ast T_X)=c_1(T_{U/M})\cdot c_{k-1}(N_{U/M\times X})+c_k(N_{U/M\times X}),
\]
by \eqref{formula_Chern_TX} this equality can be rewritten as 
\begin{equation}
\label{eq_whitey-sum}
c_k+c_{k-1}\cdot(2\alpha-\beta)=d^k\alpha^k\tag{$*_k$}
\end{equation}
where $c_k:=c_k(N_{U/M\times X})$. From $(*_{k+1})-d\alpha\cdot(*_k)$ we get
\[
(c_{k+1}-d\alpha\cdot c_k)=(c_k-d\alpha\cdot c_{k-1})\cdot(-2\alpha+\beta).
\]
Thus recursively we deduce that
\[
c_k-d\alpha\cdot c_{k-1}=(-2\alpha+\beta)^{k-1}\cdot (c_1-d\alpha),
\]
that is,
\[
c_k=d\alpha\cdot c_{k-1}+(-2\alpha+\beta)^k\,,
\]
and consequently
\begin{equation}
c_k=\sum_{i=0}^k (-2\alpha+\beta)^{k-i}\cdot(d\alpha)^i.
\end{equation}
Combine this with ({\hyperref[eq_whitey-sum]{$*_{n}$}}) we get 
\[
d^n\alpha^n=-c_{n-1}\cdot(-2\alpha+\beta)=-\sum_{i=1}^n(-2\alpha+\beta)^i\cdot(d\alpha)^{n-i},
\]
that is 
\begin{equation}
\label{eq_vanishing}
\sum_{i=0}^n (-2\alpha+\beta)^{n-i}\cdot(d\alpha)^{i}=0.
\end{equation}

Since $N_{U/M\times X}$ is a nef vector bundle, by \cite[Corollary 2.6]{DPS94} we have 
\begin{equation}
\label{ineq_nef-Segre2}
(c_1^2-c_2)\cdot S=d\alpha\cdot(-2\alpha+\beta)\cdot S\geqslant 0
\end{equation}
for every proper surface $S\subseteq U$. Let $H$ be any (irreducible) hypersurface in the base point free linear system $\left|-me^\ast K_X\right|$ for $m$ sufficiently large and divisible, then for any curve $C$ contained in $H$, set $S_C:=\pi\inv(\pi(C))$ and then by the inequality \eqref{ineq_nef-Segre2} we obtain
\[
md(-2\alpha+\beta)|_H\cdot C=md\alpha\cdot(-2\alpha+\beta)\cdot S_C\geqslant 0,
\]
which implies that $(-2\alpha+\beta)|_H$ is nef. From \eqref{eq_vanishing} we have
\[
\sum_{i=0}^n(-2\alpha+\beta)^{n-i}\cdot(d\alpha)^i\cdot H=0
\]
but since $(-2\alpha+\beta)|_H$ is nef we must have
\[
md^{i+1}(-2\alpha+\beta)^{n-i}\cdot \alpha^{i+1}=(-2\alpha+\beta)^{n-i}\cdot(d\alpha)^i\cdot H=0.
\]
In particular, when $i=n-1$, we get $(-2\alpha+\beta)\cdot\alpha^n=0$, which implies that $(-2\alpha+\beta)|_{U_x}\equiv 0$ for any $x\in X$, here $U_x$ denotes the fibre of $e$ over $x$, which is a smooth projective variety by {\hyperref[prop_fam-rc-CP-mfd]{Proposition \ref*{prop_fam-rc-CP-mfd}(b)}}. But $\alpha|_{U_x}\sim_{\QQ}0$, hence $\beta|_{U_x}\equiv0$ and in particular $\det N_{U/M\times X}|_{U_x}\equiv 0$. Since $N_{U/M\times X}$ is a nef vector bundle, $N_{U/M\times X}|_{U_x}\simeq N_{U_x/M}$ is numerically flat, where the isomorphism follows from {\hyperref[prop_fam-rc-CP-mfd]{Proposition \ref*{prop_fam-rc-CP-mfd}(d)}}. In particular, by {\hyperref[lemma_h0-num-flat]{Lemma \ref*{lemma_h0-num-flat}}} we have 
\begin{equation}
\label{ineq_num-flat}
\dimcoh^0(U_x\,,N_{U_x/M})\leqslant n-1.
\end{equation}

Now consider the Hilbert scheme of $M$. By {\hyperref[prop_fam-rc-CP-mfd]{Proposition \ref*{prop_fam-rc-CP-mfd}(d)}}, $e:U\to X$ is a family of subvarieties of $M$, hence there is a unique morphism $g: X\to\Hilb(M)$ such that $U$ equals to the pullback of the universal family via $g$. Let $T$ be a component of $\Hilb(M)$ that contains the image of $X$, and let $u: U_T\to T$ be the universal family over $T$, then we have the following commutative diagram
\begin{center}
\begin{tikzpicture}[scale=2.5]
\node (A) at (0,0) {$T$,};
\node (B) at (-1,0) {$X$};
\node (A1) at (0,1) {$U_T$};
\node (B1) at (-1,1) {$U$};
\node (C) at (1,1) {$M$};
\path[->,font=\scriptsize,>=angle 90]
(A1) edge node[right]{$u$} (A)
(B1) edge node[left]{$e$} (B)
(B) edge node[below]{$g$} (A)
(B1) edge node[below]{$g_U$} (A1)
(A1) edge node[below]{$h$} (C)
(B1) edge [bend left] node[above]{$\pi$} (C);
\end{tikzpicture}  
\end{center}
where $h:U_T\to M$ is the evaluation morphism, and the left square is Cartesian. By \cite[I.2.8, p.~31]{Kollar96} and \eqref{ineq_num-flat} we have $\dim T\leqslant n-1$; since $u$ is a flat fibration of relative dimension $d-2$ (cf. {\hyperref[prop_fam-rc-CP-mfd]{Proposition \ref*{prop_fam-rc-CP-mfd}(b)}}), we have $\dim U_T\leqslant n+d-3$. Since $\pi$ is surjective, so is $h$, but $\dim M=n+d-3\geqslant\dim U_T$, hence a fortiori $\dim T=n-1$, $\dim U_T=n+d-3$ and $h$ is generically finite; moreover, since $\pi$ is a fibration, so is $h$, hence $h$ must be a birational morphism. In particular, a general fibre of $g_U$ is isomorphic to $\PP^1$. By applying Chevalley's semicontinuity theorem (\cite[Lemma (3.1.1), p.188]{EGA4-3}) to $g_U$ we see that $h$ must be a finite morphism, thus an isomorphism by Zariski's Main Theorem (q.v.~\cite[Theorem 1.11, pp.9-10]{Uen75}). Hence $g$ is also a $\PP^1$-fibration, and it is at the same time the contraction of the extremal ray $R_i$ (see~{\hyperref[prop_cone-CP-mfd]{Proposition \ref*{prop_cone-CP-mfd}}}). In consequence, every elementary extremal contraction of $X$ is a $\PP^1$-fibration, thus $X$ is of flag type (see~{\hyperref[sec_setup]{\S \ref*{sec_setup}}}), contradicting to our {\hyperref[assumption_dim-FT]{Assumption \ref*{assumption_dim-FT}}}.  
\end{proof}

%% file: V1/cover_rc.tex
\section{Uniruledness of the augmented base locus of \texorpdfstring{$L$}{text}}
\label{sec_uniruled}

This section constitutes the key step towards the proof of {\hyperref[mainthm_DPP]{Theorem \ref*{mainthm_DPP}}}. Indeed we will prove the following result, which, roughly speaking, states that $\BB_+(L)$ is covered by rational curves intersecting $L$ trivially:

\begin{thm}
\label{thm_cover-rc}
Let everything as in {\hyperref[sec_setup]{\S \ref*{sec_setup}}}. Then for any general point $x$ in the augmented base locus $\BB_+(L)$ of $L$\,, there is a rational curves $\Gamma_x\subseteq\BB_+(L)$ passing through $x$ with $L\cdot\Gamma_x=0$\,.
\end{thm}

The initial idea to formulate this result is inspired by the following theorem of Caucher Birkar and the questions posed in the subsequent work \cite[\S 6]{Bir17}.

\begin{thm}[{\cite[Theorem 1.11]{Bir16}}]
\label{thm_Birkar}
Let $X$ be a normal projective variety of dimension $d$ over an algebraically closed field (of any characteristic), and let $B$ and $A$ be effective $\RR$-divisors. Suppose $A$ is nef and big and $L:=K_X+B+A$ is nef. If $L^d=0$, then for each general closed point $x\in X$, there is a rational curve $\Gamma_x$ passing through $x$ with $L\cdot \Gamma_x=0$\,.
\end{thm}

The subtle point here is that the uniruledness result does not hold in general for $\BB_+(L)$ if $L$ is big, as illustrated by the `counterexamples'  of \cite[Examples 6.2 and 6.3]{Bir17}. More specifically, these examples show that there is strong obstruction that prevents divisorial components of $\BB_+(L)$ from being covered by $L$-trivial rational curves; in consequence it is a crucial step (in {\hyperref[sec_codim]{\S \ref*{sec_codim}}}) to rule out the possibility that $\BB_+(L)$ has codimension $1$ in $W$ (see {\hyperref[prop_codim-B+]{Proposition \ref*{prop_codim-B+}}}).  
%Before entering into the proof of {\hyperref[thm_cover-rc]{Theorem \ref*{thm_cover-rc}}}, let us emphasize that it is a crucial step (in {\hyperref[sec_codim]{\S \ref*{sec_codim}}}) to rule out the possibility that $\BB_+(L)$ has codimension $1$ in $W$ (see {\hyperref[prop_codim-B+]{Proposition \ref*{prop_codim-B+}}}), otherwise we might encounter the same difficulty exhibited in the `counterexamples'  of \cite[Examples 6.2 and 6.3]{Bir17}.

\paragraph{}
The rest of the section is devoted to the proof of {\hyperref[thm_cover-rc]{Theorem \ref*{thm_cover-rc}}}, which is inspired by \cite{BBP13}. First we set up some notations. By \cite[Vol.II, \S 10.3, Lemma 10.3.1, p.~247; Vol.I, \S 2.1.A, Proposition 2.1.21, p.~127]{Laz04}, we have
\[
\BB_+(L)=\BB(L-\epsilon P)=\Bs\left|k\mu_0(L-\epsilon P)\right|_{\red}
\]
for $\epsilon\in\QQ_{>0}$ sufficiently small, for some $m_0\in\ZZ_{>0}$ such that $\mu_0\epsilon\in\ZZ$ and for every $k\in\ZZ_{>0}$ sufficiently large. Now take $k_0$ sufficiently large so that $k_0\mu_0\epsilon\geqslant n$ and that the rational mapping $|B_0|:W\dashrightarrow \PP\Coh^0(W,\scrO_W(B_0))$ is birational to the image, where
\[
B_0:=k_0\mu_0(L-\epsilon P).
\]
Moreover, from the definition of restricted base loci we see that
\[
\BB_+(L)=\BB(L-\epsilon P)\supseteq\BB_-(L-\epsilon P)=\BB(L-\epsilon'P)=\BB_+(L),
\]
for any positive rational number $\epsilon'<\epsilon$, hence 
\begin{equation}
\label{eq_B=B-}
\BB_+(L)=\BB(L-\epsilon P)=\BB_-(L-\epsilon P).
\end{equation}
By {\hyperref[prop_codim-B+]{Proposition \ref*{prop_codim-B+}}}, 
\[
\codim\Bs|B_0|=\codim\BB_+(L)\geqslant 2,
\]
hence the linear system $|B_0|$ is mobile; moreover, 
%by {\hyperref[assumption_dim-FT]{Assumption \ref*{assumption_dim-FT}}}, 
$\dim W=n^2+2n\geqslant 3$, then by Bertini irreducibility theorem (q.v.~\cite[Theorem 10.1.14, p.170]{KLOS21}), a general member of $|B_0|$ is irreducible, and it is reduced since it is generically smooth (smooth outside $\Bs|B_0|$) by Bertini smoothness theorem (q.v.~\cite[Theorem 10.1.2, p.167]{KLOS21}). 

Let $S$ be an irreducible component of $\BB_+(L)=\Bs|B_0|$ and set $c:=\codim S$ ($c\geqslant 2$ by {\hyperref[prop_codim-B+]{Proposition \ref*{prop_codim-B+}}}). In the sequel we will prove that $S$ is covered by rational curves that intersect trivially with $L$\,.

\subsection{Study of the singularities of the \texorpdfstring{$\QQ$}{text}-linear system \texorpdfstring{$|B_0|_{\QQ}$}{text}}
\label{ss_sing-B0}
As a prelude to the proof of {\hyperref[thm_cover-rc]{Theorem \ref*{thm_cover-rc}}}, we study the singularities of the $\QQ$-linear system $|B_0|_{\QQ}$, and our attention is particularly paid to the asymptotic order of vanishing of $B_0$ along $S$. In order to understand its relation to {\hyperref[thm_cover-rc]{Theorem \ref*{thm_cover-rc}}}, we consider the following heuristic situation:

Take $B_1,\cdots, B_c$ be $c$ general members of $|B_0|$\,, and put
\[
B:=B_1+\cdots+B_c\,,
\]
and \textbf{suppose that $(W,B)$ is dlt at the generic point of $S$}.
In this case, $S$ is an irreducible component of $B_1\cap\cdots\cap B_c$\, and $(W,B)$ is SNC at the generic point of $S$, thus by definition (q.v.~\cite[Definition 4.15, p.163]{Kollar13}) $S$ is an lc centre of $(W,B)$ (cf.~\cite[Theorem 4.16, p.164]{Kollar13}). By subadjunction \cite[Theorem 4.19 and Complement 4.19.1, pp.166-167]{Kollar13}, there is an effective divisor $\Delta_{\bar S}$ on the normalization $\bar S$ of $S$, such that 
\[
(K_W+B)|_{\bar S}\sim_{\QQ} K_{\bar S}+\Delta_{\bar S}\,,
\]
here we use the restriction notation $|_{\bar S}$ to denote the pullback to the normalization $\bar S$ of $S$. Then put $A:=L+(ck_0\mu_0\epsilon-n)P$ and by \eqref{formulae_canonical} we have
\[
K_{\bar S}+\Delta_{\bar S}+A_{\bar S}\sim_{\QQ}(K_W+B+A)|_{\bar S}=(ck_0\mu_0-n)L_{\bar S}\,.
\]
By Nakamaye's theorem on base loci \cite[Theorem 0.3]{Nak00} (see also \cite[\S 10.3, Theorem 10.3.5, p.~248]{Laz04}) we have $L_{\bar S}^{n^2+2n-c}=L^{n^2+2n-c}\cdot S=0$, hence we can apply {\hyperref[thm_Birkar]{Theorem \ref*{thm_Birkar}}} to $(\bar S, \Delta_{\bar S}\,, A_{\bar S})$ to conclude in this case. We can also apply the argument of \cite[Proof of Theorem A]{BBP13} to cover this case.

Nevertheless, in general we cannot expect the singularities of $|B_0|$ along $S$ to be so mild. Indeed, as we take $k_0$ to be larger and larger, $|B_0|$ gets more and more singular along $S$ (e.g. the multiplicity gets larger and larger), whilst we do not have good control on $k_0$ (or $\mu_0$). In order to bypass this difficulty, a natural idea is to take a dlt modification of $(W,B)$ after suitably choosing the boundary divisor $B\geqslant 0$, and establish the uniruledness result for $S$ by running certain MMP as in \cite[Proof of Theorem A]{BBP13}. The subtle point here is: in order to make this argument to work, we should guarantee that the exceptional divisor on the dlt model that dominates $S$ is contracted by this MMP, and this requires a study of the singularities of $|B_0|_{\QQ}$ along $S$. More precisely, we need the following:

\begin{lemma}
\label{lemma_sing-B0}
Let everything as above. Then $b(n+1)\leqslant c$. 
\end{lemma}
Before giving the proof of the lemma, let us first explain the definition of the function $b(\cdot)$: for any $\lambda\in\QQ_{>0}$, set 
\[
b(\lambda):=\lim_{k\to+\infty}\sigma_S(kL-\lambda P),
\]
here the limit exists since for a fixed $\lambda$, $\sigma_S(kL-\lambda P)$ is decreasing in $k$ by subaddivity of $\sigma_S$ ({\hyperref[prop_AOV]{Proposition \ref*{prop_AOV}(b)}}) and nefness of $L$. Of course, one can extend $b(\cdot)$ to a function defined on $\RR_{>0}$, but this is not needed in our paper. Moreover, we see that $b(\cdot)$ is a linear function: indeed, for $a\in\ZZ_{>0}$ by homogeneity of $\sigma_S$ ({\hyperref[prop_AOV]{Proposition \ref*{prop_AOV}(a)}}) we have
\[
b(a\lambda)=\lim_{k\to+\infty}\sigma_S(kL-a\lambda P)=\lim_{k'\to+\infty}\sigma_S(a(k'L-\lambda P))=a\cdot b(\lambda).
\]

%\input{V1/Pf_lemma-sing}

%\begin{comment}
\begin{proof}[{Proof of {\hyperref[lemma_sing-B0]{Lemma \ref*{lemma_sing-B0}}}}]
Otherwise if $b(n+1)>c$, then consider the linear series $\scrB:=|(n+1)L_1-(n+1)P|$ which is non-empty by {\hyperref[prop_L+-tP]{Proposition \ref*{prop_L+-tP}(b)}} and we have
\[
\mult_S\scrB\geqslant\sigma_S((n+1)L_1-(n+1)P)\geqslant\sigma_S((n+1)L-(n+1)P)\geqslant b(n+1)>c,
\]
thus $\mult_S\scrB\geqslant c+1$. Set $\scrJ:=\scrJ(\scrB)$ and let $Z:=Z(\scrJ)$ be the subscheme of $W$ defined by $\scrJ$, then by 
%\cite[Vol.II, \S 9.3.A, Example 9.3.5 or Example 9.3.7, p.164]{Laz04} 
{\hyperref[prop_multiplier-multiplicity]{Proposition \ref*{prop_multiplier-multiplicity}}} we have $\scrJ\subseteq\scrI_S^{\langle 2\rangle}$. 
%(see also {\hyperref[ss_pre_multiplier-ideals]{\S \ref*{ss_pre_multiplier-ideals}}})
In the sequel we proceed in 2 steps to get a contradiction.

\paragraph{Step 1.} We show that $\pi_{1\ast}(\scrO_W(L)\otimes\scrJ)\neq 0$.

To this end, it suffices to show that for $m\in\ZZ_{>0}$ sufficiently large, $q_\ast(\scrO_W(L+mL_1)\otimes\scrJ)\neq 0$, which can be guaranteed by the non-vanishing of $\Coh^0(W_x\,,\scrO_{W_x}(L+mL_1)\otimes\scrJ(\scrB_x))$
for general $x\in X$, here we set $\scrB_x:=\scrB|_{W_x}\subseteq\left|(n+1)L_1|_{W_x}\right|$ and we use the fact that $\scrB_x$ is not empty and $\scrJ|_{W_x}=\scrJ(\scrB_x)$ for general $x\in X$ (q.v.~\cite[Vol.II, \S 9.5.D, Example 9.5.37(b), p.211]{Laz04}). Now consider the natural projection 
\[
\pr_1: W_x\simeq(\PP^n)^{\times(n+1)}\to V_x\simeq\PP^n.
\]
An easy calculation gives $K_{W_x/V_x}=-(n+1)L_1|_{W_x}$\,, thus by writing
\[
mL_1|_{W_x}=K_{W_x/V_x}+(n+1+m)L_1|_{W_x}\,,
\]
we get $\RDer^i\!\pr_{1\ast}(\scrO_{W_x}(mL_1)\otimes\scrJ(\scrB_x))=0$ for $i>0$ by the relative Nadel vanishing ({\hyperref[prop_rel-Nadel-vanishing]{Proposition \ref*{prop_rel-Nadel-vanishing}}}), and this implies that we have the following short exact sequence
\begin{equation}
\label{ses_scrJ_x}
0\to\scrF_m\to \scrO_{V_x}\otimes\Coh^0((\PP^n)^{\times n},\boxtimes_{i=1}^n\scrO_{\PP^n}(m))\to \scrG_m\to 0,
\end{equation}
where $\scrF_m:=\pr_{1\ast}(\scrO_{W_x}(mL_1)\otimes\scrJ(\scrB_x))$ and $\scrG_m:=\pr_{1\ast}(\scrO_{Z_x}(mL_1))$. From \eqref{ses_scrJ_x} we see that the torsion free part of $\scrG_m$ (that is, the quotient of $\scrG_m$ by its torsion subsheaf) admits a singular Hermitian metric that is Griffiths semipositive, in particular $\det\scrG_m$ is pseudoeffective; on the other hand, by {\hyperref[thm_pos-direct-image]{Theorem \ref*{thm_pos-direct-image}}} the natural $L^2$-metric on $\scrF_m$ is Griffiths semipositive and $L^2$-reflexive, which implies that $\det\scrF_m$ is pseudoeffective. But from \eqref{ses_scrJ_x} we have  
\[
\det\scrF_m\otimes\det\scrG_m\simeq\scrO_{V_x}\,,
\]
hence $\det\scrF_m\equiv\det\scrG_m\equiv0$. It follows that $\scrF_m$ (equipped with the natural $L^2$-metric) is a Hermitian flat vector bundle by 
%{\hyperref[thm_pos-direct-image]{Theorem \ref*{thm_pos-direct-image}}} and 
{\hyperref[prop_pos+det=0]{Proposition \ref*{prop_pos+det=0}}}. Moreover, by \cite[Theorem 1.20 and Proposition 1.16]{DPS94} we see that $\scrF_m$ is a subbundle of the trivial vector bundle $\scrO_{V_x}\otimes\Coh^0((\PP^n)^{\times n},\boxtimes_{i=1}^n\scrO_{\PP^n}(m))$; in consequence $\scrG_m$ is locally free and thus a Hermitian flat vector bundle. Since $\PP^n$ is simply connected, flat vector bundles $\scrF_m$ and $\scrG_m$ are both trivial, hence 
%Thus by {\hyperref[cor_Pn-h0-num-flat]{Corollary \ref*{cor_Pn-h0-num-flat}}} we get
\[
\dimcoh^0(W_x\,,\scrO_{Z_x}(L+mL_1))=\dimcoh^0(V_x\,,\scrG_{m+1}\otimes\scrO_{V_x}(H))=(n+1)\cdot\rank\scrG_{m+1}=O(m^{n^2-1}).
\]
By Nadel vanishing \cite[Vol.II, \S 9.4.B, Corollary 9.4.15, p.190]{Laz04} we have 
\[
\Coh^1(W_x\,,\scrO_{W_x}(L+mL_1)\otimes\scrJ(\scrB_x))=\Coh^1(W_x\,,\scrO_{W_x}(K_{W_x}+(n+2)L+mL_1)\otimes\scrJ(\scrB_x))=0
\]
hence for $m\gg1$ we obtain
\begin{align*}
\dimcoh^0(W_x\,,\scrO_{W_x}(L+mL_1)\otimes\scrJ(\scrB_x)) &=\dimcoh^0(W_x\,,\scrO_{W_x}(L+mL_1))-\dimcoh^0(W_x\,,\scrO_{Z_x}(L+mL_1)) \\
&=(n+1)\cdot\binom{m+n}{n}^{n}+O(m^{n^2-1})>0,
\end{align*}
and this implies that  $\pi_{1\ast}(\scrO_W(L)\otimes\scrJ)\neq 0$.

\paragraph{Step 2.} For general $t\in\pi_1(S)$ we show that $\Coh^0(W_t\,,\scrO_{W_t}(L)\otimes\scrI_{S_t}^{\langle 2\rangle})\neq 0$ and conclude.

For any $m\in\ZZ_{>0}$ set
\[
\scrH_m:=\pi_{1\ast}(\scrO_W(L+mL_1+mP)\otimes\scrJ)\,,
\]
then from {\bf Step 1} it follows that $\scrH_m\neq 0$. Since $L_1+P$ is the pullback of an ample divisor on $V_n$ via $\pi_1$, $\scrH_m$ is globally generated for $m\gg 1$. Since $\scrB=\left|(n+1)L_1-(n+1)P\right|$ is the pullback of a linear series on $V_n$ via $\pi_1$\,, $\pi_1|_Z:Z\to V_n$ is not dominant, in particular, for general $t\in V_n$\,, $\scrJ|_{W_t}=\scrO_{W_t}$ and thus
\begin{equation}
\label{eq_scrH-general-pt}
\scrH_m\otimes\kappa(t)\simeq\Coh^0(W_t\,,\scrO_{W_t}(L));
\end{equation}
moreover, since $\scrH_m$ is globally generated, $\pi_1|_{\Bs|\Lambda_m|}:\Bs|\Lambda_m|\to V_n$ is not dominant where the linear series
\[
\Lambda_m:=\Coh^0(W,\scrO_W(L+mL_1+mP)\otimes\scrJ).
\]
Now set $F:=\Fix|\Lambda_m|$, then $F=\pi_1^\ast F_0$ for some divisor $F_0\geqslant 0$ on $V_n$ since $\pi_1$ is a smooth fibration. Again since $L_1+P$ is the pullback of an ample divisor on $V_n$\,,  $F+m'(L_1+P)$ is globally generated for $m'\gg 1$\,, and thus $|\Lambda_{m+m'}|$ is mobile. In other word, up to enlarging $m$ we have that $\codim\Bs|\Lambda_m|\geqslant 2$. 

Again by \eqref{eq_scrH-general-pt} and global generation of $\scrH_m$ we see that the rational map defined by $\Lambda_m$ maps $W_t$ isomorphically to its image for general $t\in V_n$\,, in particular, the image of $|\Lambda_m|$ has dimension $\geqslant n\geqslant 2$. Hence by Bertini irreducibility theorem (q.v.~\cite[Theorem 10.1.14, p.170]{KLOS21}), a general divisor $D\in|\Lambda_m|$ is irreducible. \textcolor{red}{By looking at the restriction of $D$ to a general fibre of $\pi_1$ we find that $D$ is horizontal and thus equidimensional (q.v.~\cite[Lemma 1.6.2]{Wang-thesis}) over $V_n$\,, in particular $D$ does not contain any fibre of $\pi_1$} and thus for any $t\in V_n$\,, the restriction $D|_{W_t}\in\left|\scrO_{W_t}(L)\right|\simeq|\scrO_{\PP^n(1)}|$ is smooth (since every hyperplane in $\PP^n$ is smooth). On the other hand, we have already seen that $\scrJ\subseteq\scrI_S^{\langle 2\rangle}$\,, hence $D$ vanishes of order $\geqslant 2$ along $S$, and consequently for general $t\in\pi_1(S)$, $D|_{W_t}$ vanishes of order $\geqslant 2$ along every component of $S_t$\,, which contradicts the smoothness of $D|_{W_t}$ and the lemma thus is proved.
\end{proof}

\subsection{Proof of \texorpdfstring{{\hyperref[thm_cover-rc]{Theorem \ref*{thm_cover-rc}}}}{text}}
\label{ss_proof-cover-rc}
With preparations accomplished in {\hyperref[ss_sing-B0]{\S \ref*{ss_sing-B0}}} we are now at the point to give the:

\begin{proof}[Proof of {\hyperref[thm_cover-rc]{Theorem \ref*{thm_cover-rc}}}]

As pointed out above, the idea of the proof is to choose an appropriate boundary $B$ on $W$ and take a dlt modification of $(W,B)$, on which we run a certain MMP, as inspired by \cite[Proof of Theorem A]{BBP13}; the key point here is to prove that the exceptional divisor lying over $S$ is contracted by this MMP, and then we are done since the curves contracted by the MMP are $K$-negative, thus cannot be contracted by the dlt modification which is a relative minimal model. The proof proceeds in 4 steps.

\paragraph{Step 0.} We make some preparations.

Let $k_0$ sufficiently large such that $b(k_0\mu_0\epsilon)>1$ and take $B_1\,,\cdots,B_c$ be general members of $\left|\mu B_0\right|$ for $\mu$ sufficiently large such that for every $j$ we have 
\[
\frac{1}{\mu}\mult_S(B_j)\leqslant b(k_0\mu_0\epsilon)+\delta
\]
where $\delta$ is sufficient small (we will see later that it suffices to take $\delta<1/(4n+4)$). Here let us recall that the function $b$ is defined as follows:
\[
b(\cdot):\ZZ_{>0}\to \QQ_{>0}\,,\quad\lambda\mapsto\lim_{k\to+\infty}\sigma_S(kL-\lambda P).
\]
Now set
\[
B:=\frac{1}{\mu}\sum_{j=1}^c B_i\,,
\]
then $(W,B)$ is not lc at the generic point of $S$ since by \cite[Lemma 2.29, pp.52-53]{KM98} we have
\[
a(v_S\,,W,B)=c-1-\frac{1}{\mu}\sum_{j=1}^c\mult_S B_j\leqslant c-1-c\cdot b(k_0\mu_0\epsilon)<-1,
\]
where $v_S$ is the divisorial valuation centred at $S$ and $a(\;\cdot\;,W,B)$ denotes the discrepancy of $(W,B)$ at a divisorial valuation. 

Take a $\QQ$-factorial dlt model $f:(W',B')\to (W,B)$ of $(W,B)$ (q.v.~\cite[Theorem 1.34, p.28]{Kollar13}). Here, each $B_i$ is irreducible by Bertini, hence $B$ is a boundary (q.v.~\cite[Definition 1.3, p.6]{Kollar13}), i.e. every coefficient of $B$ is in $[0,1]$. By construction we have $B'=f_\ast^{-1}B+E$ with $E:=\Exc(f)$ and $K_{W'}+B'$ is $f$-nef. Write
\begin{equation}
\label{eq_log-discrepancy}
K_{W'}+B'\sim_{\QQ}f^\ast(K_W+B)+\sum_i a_i E_i\,,
\end{equation}
by construction $a_i=1+a(E_i,W,B)=\text{log-discrepancy of }(W,B)\text{ at }E_i$\,, and by negativity lemma (q.v~\cite[Lemma 3.39, pp.102-103]{KM98}) $a_i\leqslant 0$ for every $i$. Since $(W,B)$ is not dlt at the generic point of $S$, there is an irreducible component of $E$, say $E_1$, that dominates $S$. Again by \cite[Lemma 2.29, pp.52-53]{KM98} we get
\begin{equation}
\label{ineq_a_1}
a_1=1+(c-1)-\frac{1}{\mu}\sum_{j=1}^{c}\mult_S B_j\geqslant c-b(k_0\mu_0\epsilon c)-c\delta.
\end{equation}
Note that every lc centre of $(W',B')$ is normal by \cite[Theorem 4.16, p.164]{Kollar13} and in particular $E_1$ is normal. In the sequel we will prove that $E_1$ is covered by rational curves that intersect trivially with $f^\ast L$ and are not contracted by $f$. 

Set 
\[
A:=\frac{1}{2}(L+P),
\]
then $A$ is ample on $W$. Take $G\geqslant 0$ to be an $f$-exceptional $\QQ$-divisor such that $-G$ is $f$-ample. We choose $G$ sufficiently small so that $G\leqslant\delta\cdot E$ and $A':=f^\ast A-G$ is ample on $W'$, and we set
\[
A'_m:=\frac{1}{k}\cdot\text{general member of the linear system }\left|k(A'+mf^\ast L)\right|\,,
\]
where $m\in\ZZ_{>0}$ and $k$ sufficiently large and divisible. Then $(W',B'+A'_{m})$ is still a dlt pair, and we set 
\[
D_m:=K_{W'}+B'+A'_m\,.
\] 

\paragraph{Step 1.} We run a $D_m$-MMP with scaling of a fixed sufficiently ample divisor for $m\gg1$.

The existence and termination of such MMP is essentially guaranteed by {\hyperref[thm_lc-flip]{Theorem \ref*{thm_lc-flip}}}. Here the subtle point is that for $m$ sufficiently large, we can make this MMP independent of $m$ and $f^\ast L$-trivial. To this end, we proceed as follows.

Set $W_0:=W'$, by {\hyperref[thm_lc-cone]{Theorem \ref*{thm_lc-cone}}}, there are only finitely many $D_m$-negative extremal rays and each one is spanned by a rational curve that has intersection number $\leqslant 2(n^2+2n)$ with $-D_m$, then we see that there is $m_0\in\ZZ_{>0}$ such that for any $m\geqslant m_0$\,, $\NEb(W')_{D_m<0}$ is independent of $m$ and are spanned by finitely many extremal $D_m$-negative rays that intersect trivially with $f^\ast L$. Hence for any $m\geqslant m_0$\,, by {\hyperref[thm_lc-flip]{Theorem \ref*{thm_lc-flip}(a)}} we get the first step of this $D_m$-MMP, denoted by $g_1: W_0\dashrightarrow W_1$, which is independent of $m$ and is $f^\ast L$-trivial; and by \cite[Theorem 3.7(4), p.76]{KM98} there is a nef divisor on $L_{W_1}$ on $W_1$ such that $g_1^\ast L_{W_1}=f^\ast L$\,. By the same argument as above applied to $W_1$, up to enlarging $m_0$, for any $m\geqslant m_0$ we get the second step of this $D_m$-MMP $g_2:W_1\dashrightarrow W_2$\,, which is independent of $m$ and $L_{W_1}$-trivial, and a nef divisor $L_{W_2}$ on $W_2$ such that $g_2^\ast L_{W_2}=L_{W_1}$. Continue this process, we get a series of divisorial contractions or flips $g_i: W_{i-1}\dashrightarrow W_{i}$ that are independent of $m$ ($m\geqslant m_0$) and $L_{W_{i-1}}$-trivial, and this sequence must terminate by {\hyperref[thm_lc-flip]{Theorem \ref*{thm_lc-flip}(b)}}; in this way we obtain a $D_m$-MMP (with scaling of a sufficiently ample divisor), denoted by $g: W'\dashrightarrow W_l:=Z$\,, that is independent of $m$ ($m\geqslant m_0$) and $f^\ast L$-trivial. Moreover, we take an elimination of indeterminacies of $g$ as follows:
\begin{center}
\begin{tikzpicture}[scale=2.0]
\node (A) at (0,0) {$W$};
\node (B) at (0,1) {$W'$};
\node (C) at (2,1) {$Z$,};
\node (D) at (1,2) {$\tilde{Z}$};
\path[->,font=\scriptsize,>=angle 90]
(D) edge node[above left]{$\phi$} (B)
(D) edge node[above right]{$\tilde{g}$} (C)
(B) edge node[left]{$f$} (A);
\path[dashed,->,font=\scriptsize,>=angle 90]
(B) edge node[below]{$g$} (C);
\end{tikzpicture}  
\end{center}
then $g_\ast D_m:=\tilde g_\ast\phi^\ast D_m$ is nef ($g$ is a $D_m$-MMP) and there is a nef divisor $L_Z$ on $Z$ such that $\phi^\ast f^\ast L\sim\tilde{g}^\ast L_{Z}$\,. 

\paragraph{Step 2.} We prove that $E_1$ is contracted by $g$.

%If $E_1$ is contracted by $g=g_1$, then there is nothing to prove, otherwise we have the following stronger result:
%\begin{lemma}
%\label{lemma_E1-contrated}
%If $E_1$ is not contracted by $g$, then it is contracted by $g_t$ for every $t\in\,]\,0,1[$\,.  
%\end{lemma}
%\begin{proof}
%Suppose that $E_1$ is not contracted by $g$, and fix a $t\in\,]\,0,1[$\,, we will show that $E_1$ is contracted by $g_t$\,. By construction (and Bertini) $A'_m$ is ample and irreducible with coefficient $<1$, then from {\hyperref[lemma_lc-centre-MMP]{Lemma \ref*{lemma_lc-centre-MMP}(b)}} we see that $\BB_+(A'_{m,Z})$ does not contain any lc centre of $(Z,B'_Z+A'_{m,Z})$, hence we can write $A'_{m,Z}=H+\Delta$ with $H$ ample and $\Delta\geqslant 0$ does not contain any lc centre of $(Z,B'_Z+A'_{m,Z})$\,. On the other hand, since $E_1$ is not contracted by $g$, $E_{1,Z}$ is a component of $B'_Z$ with coefficient $1$, in particular $E_{1,Z}$ is an lc centre of $(Z,B'_Z+A'_{m,Z})$, and consequently $E_{1,Z}$ is not a component of $\Delta$, in other word $\Delta|_{E_{1,Z}}\geqslant 0$.

%By construction, $D_{m,Z}=K_Z+B'_Z+A'_{m,Z}$ is nef; moreover, $D_{m,Z}|_{E_{1,Z}}$ is not big, otherwise,  $D_m|_{E_1}$ is also big (q.v.~\cite[]{Laz04}), and so is $f^\ast L|_{E_1}$ 

%\end{proof}

The key point of this step is to show that $\sigma_{E_1}(D_m)>0$. Since 
\[
\sigma_{E_1}(\sum_{i\neq 1}(-a_i)E_i+(G-\mult_{E_1}G\cdot E_1))=0, 
\]
by \eqref{eq_log-discrepancy} and subadditivity of $\sigma_{E_1}$ we get
\[
\sigma_{E_1}(D_m)\geqslant\sigma_{E_1}(f^\ast(K_W+B+A+mL)-(-a_1+\mult_{E_1}G)\cdot E_1).
\]
On the hand, by {\hyperref[lemma_sing-B0]{Lemma \ref*{lemma_sing-B0}}} and by linearity of the function $b(\cdot)$ we get:
\begin{align*}
\sigma_{E_1}(f^\ast(K_W+B+A+mL)) &=\sigma_S(K_W+B+A+mL) \\
&=\sigma_S((m+k_0\mu_0c-n-\frac{1}{2})L-(k_0\mu_0\epsilon c-n-\frac{1}{2})P)\\
&\geqslant b(k_0\mu_0\epsilon c-n-\frac{1}{2}) \\
&=b(k_0\mu_0\epsilon c)-b(n+\frac{1}{2})\\
&\geqslant b(k_0\mu_0\epsilon c)-\frac{n+\frac{1}{2}}{n+1}\cdot c \\
&=b(k_0\mu_0\epsilon c)-c+\frac{c}{2n+2}\\
&> -a_1+\mult_{E_1}G,
\end{align*}
where the last inequality follows from \eqref{ineq_a_1} if we take e.g. $\delta<1/(4n+4)$ (noting that in our construction we set $G\leqslant\delta\cdot E$). Then by \cite[III.1.8, p.84]{Nak04} (though \cite{Nak04} assumes that $X$ is smooth, but it is easy to see that all the results in \S III.1 holds for $X$ normal) we have
\begin{align*}
\sigma_{E_1}(D_m) &\geqslant\sigma_{E_1}(f^\ast(K_W+B+A+mL)-(-a_1+\mult_{E_1}G)\cdot E_1)) \\
&=\sigma_{E_1}(f^\ast(K_W+B+A+mL))-(-a_1+\mult_{E_1}G)>0.
\end{align*}

Now let us prove that $E_1$ is contracted by $g$. Otherwise, if $E_1$ is not contracted by $g$, then consider $\Delta:=\phi^\ast D_m-\tilde g^\ast D_{m,Z}$, here we use the notation with subscript to denote the strict transform of a divisor on $W'$ to a birational model, e.g. $D_{m,Z}$ stands for $g_\ast D_m$. By \cite[Lemma 3.38, p.102]{KM98} $\Delta\geqslant0$ and since $E_1$ is not contracted by $g$, we see that $E_{1,\tilde Z}\not\subseteq\Supp(\Delta)$ and thus $\sigma_{E_{1,\tilde Z}}(\Delta)=0$. Then by nefness of $D_{m,Z}=g_\ast D_m$ and subadditivity of $\sigma_{E_{1,\tilde Z}}$ we get
\begin{align*}
0=\sigma_{E_{1,Z}}(D_{m,Z}) &=\sigma_{E_{1,\tilde Z}}(\tilde g^\ast D_{m,Z})\geqslant\sigma_{E_{1,\tilde Z}}(\tilde g^\ast D_{m,Z}+\Delta) \\
&=\sigma_{E_{1,\tilde Z}}(\phi^\ast D_m)=\sigma_{E_1}(D_m)>0,
\end{align*}
which is absurd, hence a fortiori $E_1$ is contracted by the $D_m$-MMP $g$. In particular, $E_1$ is covered by rational curves intersecting trivially with $f^\ast L$.

\paragraph{Step 3.} We conclude by showing that the $f^\ast L$-trivial curves constructed by Step 2 are not contracted by $f$.

Let $\lambda$ be the index such that the strict transform of $E_1$ is contracted by $g_{\lambda}:W_{\lambda-1}\to W_{\lambda}$ ($g_\lambda$ is a divisorial contraction, in particular it is everywhere defined) but not by $g_{\lambda-1}\circ\cdots\circ g_1:W_0\dashrightarrow W_{\lambda-1}$. Set $h:=g_{\lambda-1}\circ\cdots\circ g_1$ and $W'':=W_{\lambda-1}$. By \cite[Lemma 3.10.11]{BCHM10}, $h|_{E_1}:E_1\dashrightarrow E_{1,W"}$ is birational and $(W'',B_{W"})$ is dlt, in particular $E_{1,W"}$ is normal by \cite[Theorem 4.16, p.164]{Kollar13}. Set $T:=g_\lambda(E_{1_W"})$ then we have the following commutative diagram: 
\begin{center}
\begin{tikzpicture}[scale=2.0]
\node (A) at (0,0) {$W_\lambda$};
\node (A1) at (0,1) {$W"$};
\node (B) at (-1.5,0) {$W$};
\node (B1) at (-1.5,1) {$W'$};
\node (C) at (0.4,-0.4) {$T$.};
\node (C1) at (0.4,1.4) {$E_{1,W"}$};
\node (D) at (-1.9,-0.4) {$S$};
\node (D1) at (-1.9,1.4) {$E_1$};
\node (AC) at (0.2,-0.2) {\rotatebox{135}{$\subset$}};
\node (BD) at (-1.7,-0.2) {\rotatebox{45}{$\subset$}};
\node (AC1) at (0.2,1.2) {\rotatebox{45}{$\supset$}};
\node (BD1) at (-1.7,1.2) {\rotatebox{135}{$\supset$}};
\path[->,font=\scriptsize,>=angle 90]
(B1) edge node[right]{$f$} (B)
(A1) edge node[left]{$g_\lambda$} (A)
(C1) edge (C)
(D1) edge (D);
\path[dashed,->,font=\scriptsize,>=angle 90]
(B1) edge node[below]{$h$} (A1)
(D1) edge node[above]{$h|_{E_1}$} (C1);
\end{tikzpicture}  
\end{center}
For general $t\in T$, set $E_t:=h^{-1}_\ast(E_{1,W"})_t$\,, and let $f_t: E_t\to S_t$ be the Stein factorization of $f|_{E_t}:E_t\to f(E_t)\subset S$. Let $F$ be a general fibre of $f_t$\,, then $D_m|_F$ is ample since by construction $D_m$ is $f$-ample and $F$ is contained in some fibre of $f$. Since $h|_F: F\dashrightarrow F':=h_\ast F$ is birational, then $D_{m,W"}|_{F'}$ is big. On the other hand, since $-D_{m,W"}$ is $g_\lambda$-ample, and since $F'$ is contained in some fibre of $g_\lambda$, we have that $-D_{m,W"}|_{F'}$ is ample, and in consequence a fortiori $F'=\pt$ and so is $F$. In other word, $E_t$ is not contracted by $f$. Moreover, by construction $(E_{1,W"})_t$ is covered by $L_{W"}$-trivial rational curves, hence for any general $t\in T$,  $E_t$ is covered by $f^\ast L$-trivial rational curves, and the theorem is thus proved.
\end{proof}

%% file: V1/bigness_tangent.tex
\section{Bigness of the tangent bundle}

In this section we will deduce {\hyperref[mainthm_DPP]{Theorem \ref*{mainthm_DPP}}} from {\hyperref[thm_cover-rc]{Theorem \ref*{thm_cover-rc}}}. It is based on the following observation:
%\begin{thm}
%\label{thm_nef-dim-H}
%Let everything as in {\hyperref[sec_setup]{\S \ref*{sec_setup}}}. The nef dimension of $H$ is equal to $2n$, in other word, $V$ cannot be covered by curves intersecting trivially with $H$. 
%\end{thm}thm_cover-rc

\begin{lemma}
\label{lemma_non-cover-rc}
Let everything as in {\hyperref[sec_setup]{\S \ref*{sec_setup}}}. Then $V$ is not covered by rational curves that intersect trivially with $H$.
\end{lemma}
\begin{proof}
If otherwise, let $\phi:\PP^1\to V$ be a rational curve (the normalization of its image) passing through a general point of $V$ such that $\deg\phi^\ast H=0$. By \cite[II.3.11 and II.3.11.1.1, p.118]{Kollar96} $\phi$ is free in the sense of \cite[II.3.1, p.113]{Kollar96}, then by \cite[\S II.3, Definition-Proposition 3.8, p.116]{Kollar96} $\phi^\ast T_V$ is globally generated. Moreover, by \eqref{formulae_canonical} we have $-K_V\sim (n+1)H$, hence $\deg\phi^\ast(-K_V)=0$, and consequently $\phi^\ast T_V$ must be trivial over $\PP^1$, i.e. $\phi^\ast T_V\simeq\scrO_{\PP^1}^{\oplus 2n}$. On the other hand, we have a generic injection $T_{\PP^1}\to\phi^\ast T_V$, which is indeed an injective morphism (see~\cite[\S II.3.13 and its proof, p.119]{Kollar96}), and this induces a non-zero section in $\Coh^0(\PP^1,\scrO(-2)^{\oplus 2n})$, which is absurd.   
\end{proof}

\begin{proof}[Proof of {\hyperref[mainthm_DPP]{Theorem \ref*{mainthm_DPP}}}]
It suffices to show that the image of $\Delta_{V/X}$ is not contained in $\BB_+(L)$ where $\Delta_{V/X}:V\hookrightarrow W$ is the diagonal embedding. Indeed, if $\Image(\Delta_{V/X})$ is not contained in $\BB_+(L)=\Bs|B_0|_{\red}$\,, then 
\[
\left|\Delta_{V/X}^\ast B_0\right|=\left|k_0\mu_0(nH+\epsilon p^\ast K_X)\right|
\]
contains a non-zero effective divisor, by Kodaira's lemma (\cite[Vol.I, \S 2.2.A, Corollary 2.2.7, p.141]{Laz04}) this implies that $H$ is big.

Now we prove that $\Image(\Delta_{V/X})\not\subset\BB_+(L)$. Assume by contradiction that $\Image(\Delta_{V/X})$ is contained in $\BB_+(L)$, then $\pr_i|_{\BB_+(L)}:\BB_+(L)\to V$ is surjective for every $i$. Since $L$ is $\pr_i$-ample, any curve intersecting trivially with $L$ is not contracted by $\pr_i$\,, hence {\hyperref[thm_cover-rc]{Theorem \ref*{thm_cover-rc}}} implies that $V$ is covered by rational curves that intersect trivially with $H$, but this contradicts with {\hyperref[lemma_non-cover-rc]{Lemma \ref*{lemma_non-cover-rc}}}. Our main theorem is thus proved.
\end{proof}

%% file: V1/erratum.tex
\section*{Erratum}
\textcolor{red}{In Step 2 of the proof of Lemma 4.3 on page 17, the author claims that the $D$ is equidimensional over $V_n$ since it is horizontal with respect to $\pi_1$. But this claim is wrong, and thus invalidates the proof of Lemma 4.3 and that of the main theorem. Indeed, the following simple example provides a counterexample to this claim:
Consider the projection $\pr_1: X:=\PP^2\times\PP^2\to \PP^2$, the divisor $D\in|\scrO_{\PP^2\times\PP^2}(1,1)|$ defined by the homogeneous equation $x_1y_0+x_0y_1=0$ is horizontal with respect to $\pr_1$, where $[x_0:x_1:x_2]$ (resp. $[y_0:y_1:y_2]$) are homogeneous coordinates for the first (resp. second) copy of $\PP^2$ in $X$. Nevertheless, if we set $t:=[0:0:1]\in\PP^2$, then it is easy to see that $D$ contains the fibre $X_t:=\{t\}\times\PP^2$. Moreover, one can easily check that $\ord_{(t,t)}(D)=2$; indeed, if we define the linear series
\[
\Lambda:=\Coh^0(X,\scrO_{\PP^2\times\PP^2}(1,1)\otimes\mathfrak{m}^2_{(t,t)})\subseteq\Coh^0(X,\scrO_{\PP^2\times\PP^2}(1,1)),
\] 
then $\Lambda$ is generated by $x_0y_0,x_0y_1,x_1y_0,x_1y_1$ and one finds that $\Bs|\Lambda|=\{t\}\times\PP^2\cup\PP^2\times\{t\}$.}  